\documentclass[graybox]{svmult}

\usepackage{mathptmx}       
\usepackage{helvet}         
\usepackage{courier}        
\usepackage{type1cm}        
\usepackage{makeidx}         
\usepackage{graphicx}        
\usepackage{multicol}        
\usepackage[bottom]{footmisc}

\usepackage{amsmath,amsfonts,amssymb}
\usepackage{latexsym}
\usepackage{epsfig,overpic}
\usepackage{stmaryrd}

\let\eps\varepsilon
\newcommand{\bV}{\mathbf V}
\newcommand{\bVt} {\widetilde{\mathbf{V}}}
\newcommand{\R}{\mathbb R}
\newcommand{\bn}{n}
\newcommand{\bw}{\vec{w}}
\newcommand{\bx}{x}

\newcommand{\Q}{\mathcal Q}
\newcommand{\T}{\mathcal T}

\newcommand{\Div}{\operatorname{\rm div}}

\newcommand{\unabla}{{\nabla}_{\Gamma}}
\newcommand{\nablaG}{\unabla}
\newcommand{\unablah}{{\nabla}_{\Gamma_h}}
\newcommand{\rd}{\mathrm{d}}
\newcommand{\Gh}{\Gamma_h}
\newcommand{\hphi}{\phi_h^{\rm lin}}
\newcommand{\Ione}{I^1}
\newcommand{\lin}{\text{lin}}
\newcommand{\Gammalin}{\Gamma^{\lin}}

\newcommand{\OGamma}{\Omega^\Gamma}
\newcommand{\mT}{\T_h}

\newcommand{\jump}[1]{[\![#1]\!]}

\newcommand{\Vk}{V_{h,\Phi}}
\newcommand{\Vregh}{V_{\text{reg},h}}
\newcommand{\enormh}[1]{\Vert #1 \Vert_h}

\newcommand{\Gs}{\mathcal{S}}
\newcommand{\DivG}{{\,\operatorname{div_\Gamma}}}
\newcommand{\Wo}{\mbox{\small$\overset{\circ}{W}$}}
\newcommand{\la}{\left\langle}
\newcommand{\ra}{\right\rangle}
\newcommand{\bI}{\mathbf I}
\def\nat{\nabla_{\Gamma}}
\def\nath{\nabla_{\Gamma_h}}
\def\enorm#1{|\!|\!| #1 |\!|\!|}

\newtheorem{assumption}{Assumption}

\begin{document}

\title*{Trace Finite Element Methods for PDEs on Surfaces}
\author{Maxim A. Olshanskii and Arnold Reusken}

\institute{Maxim A. Olshanskii \at Department of Mathematics, University of Houston, Houston, Texas 77204-3008, USA \email{molshan@math.uh.edu}
\and Arnold Reusken \at Institut f\"ur Geometrie und Praktische Mathematik, RWTH Aachen University,
D-52056 Aachen, Germany \email{reusken@igpm.rwth-aachen.de}}

\maketitle

\abstract{In this paper we consider a class of unfitted finite element methods for discretization of partial differential equations on surfaces. In this class of methods known as the Trace Finite Element Method (TraceFEM),
restrictions or traces of background surface-independent finite element functions are used to approximate the solution of a PDE on a surface. We treat equations on steady and time-dependent (evolving) surfaces. Higher order TraceFEM is explained in detail. We review the error analysis and algebraic properties of the method. The paper navigates through the known variants of the TraceFEM and the literature on the subject.}



\section{Introduction}

Consider the Laplace--Beltrami equation on a  smooth closed surface $\Gamma$,
\begin{equation}
-\Delta_{\Gamma} u+u=f\quad\text{on}~~\Gamma.
\label{LB}
\end{equation}
Here  $\Delta_\Gamma$ is the  Laplace--Beltrami operator  on $\Gamma$.
Equation \eqref{LB} is an example of surface PDE, and it will serve as a model  problem to explain the main principles of the TraceFEM.
In this introduction we start with a brief review of the $P_1$ TraceFEM for \eqref{LB}, in  which we explain the key ideas of this method.
In this review paper  this basic $P_1$ finite element method  applied to the model problem \eqref{LB} on a stationary surface $\Gamma$ will  be extended to a general  TraceFE methodology, including   higher order elements and surface approximations, time-dependent surfaces, adaptive methods, coupled problems, etc.

The main motivation for the development of the TraceFEM is the challenge of building an accurate and computationally efficient numerical method for surface PDEs that avoids a triangulation of $\Gamma$ or any other fitting of a mesh to the surface $\Gamma$. The method turns out to be particularly useful for problems with evolving surfaces in which the surface is implicitly given by a level set function. To discretize the partial differential equation on $\Gamma$, TraceFEM uses a surface independent background mesh on a fixed  bulk  domain $\Omega\subset\mathbb{R}^3$, such that $\Gamma\subset\Omega$.  The main concept of the method is to introduce a  finite element space based on a volume triangulation (e.g., tetrahedral tessellation) of  $\Omega$, and to use traces of functions from this bulk finite element space on (an approximation of) $\Gamma$. The resulting trace space is used to define a finite element method for \eqref{LB}.

As an  example, we consider  the $P_1$ TraceFEM for \eqref{LB}. Let $\T_h$ be a consistent shape regular tetrahedral tessellation of  $\Omega \subset \Bbb{R}^3$ and  let $V_{h}^{\rm bulk}$  denote the standard FE space of continuous piecewise $P_1$ functions
w.r.t. $\T_h$. Assume $\Gamma$ is given by the zero level of a $C^2$ level set function $\phi$, i.e.,
$\Gamma=\{\bx\in\Omega\,:\, \phi(\bx)=0\}$. Consider the Lagrangian interpolant $\phi_h\in V_h^{\rm bulk}$ of $\phi$ and set
\begin{equation}\label{Gammah}
\Gamma_h:=\{\bx\in\Omega\,:\, \phi_h(\bx)=0 \}.
\end{equation}

Now we have an implicitly defined $\Gamma_h$, which is  a polygonal approximation of $\Gamma$.
This $\Gamma_h$ is a closed surface  that can be partitioned in planar triangular segments:
$
 \Gamma_h=\bigcup_{K\in\mathcal{F}_h} K,
$
where $\mathcal{F}_h$ is the set of all surface triangles.
 \begin{figure}[ht!]
 \centering
   \includegraphics[width=0.35\textwidth]{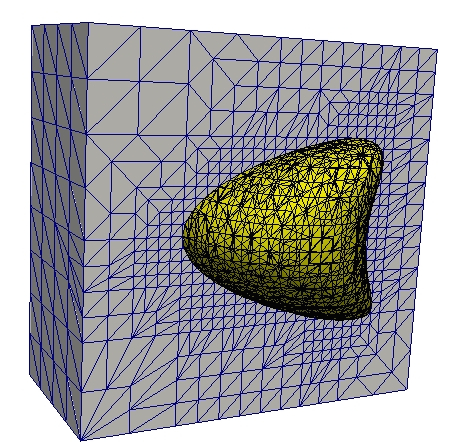}\qquad
  \includegraphics[width=0.35\textwidth]{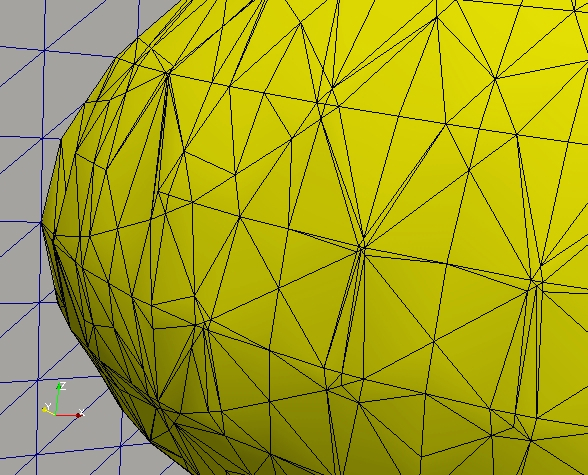}
 \caption{Example of a background mesh $\mathcal{T}_h$  and induced surface mesh $\mathcal{F}_h$.}
 \label{fig:surfaceMesh}
\end{figure}
The bulk triangulation $\T_h$, consisting of tetrahedra and the induced surface triangulation  are illustrated in Figure~\ref{fig:surfaceMesh} for   a surface from \cite{ORXsinum}. There are no restrictions on how $\Gamma_h$ cuts through the background mesh, and thus   the resulting triangulation $\mathcal{F}_h$  is \emph{not} necessarily  regular. The elements from $\mathcal{F}_h$ may have very small interior angles and the size of neighboring triangles can vary strongly, cf.~Figure~\ref{fig:surfaceMesh} (right).
Thus $\Gamma_h$ is not a ``triangulation of $\Gamma$'' in the  usual sense (an $O(h^2)$ approximation of $\Gamma$, consisting of \textit{regular} triangles).
This surface triangulation $\mathcal{F}_h$ is an easy to compute $O(h^2)$ accurate approximation of $\Gamma$ and in the TraceFEM it is used only to perform numerical integration.  The \textit{approximation properties of the method entirely depend on the volumetric tetrahedral mesh $\T_h$}.
The latter is a fundamental property of the TraceFEM, as will be explained in more detail  in the remainder of this article.

As starting point for the finite element method we use a weak formulation of \eqref{LB}:
Find $u\in H^1(\Gamma)$ such that
$
\int_{\Gamma} uv+\nabla_\Gamma u\cdot\nabla_\Gamma v \, \,ds =\int_{\Gamma}f v\, \,ds
$
for all $v\in H^1(\Gamma)$. Here $\nabla_\Gamma$ is the tangential gradient on $\Gamma$.  In the TraceFEM, in the weak formulation \textit{one replaces $\Gamma$ by $\Gamma_h$  and instead of $H^1(\Gamma_h)$ uses the space of traces on $\Gamma_h$ of all functions from the bulk finite element space.}
The Galerkin formulation of \eqref{LB} then reads:  Find $u_h\in V_h^{\rm bulk}$ such that
\begin{equation}
\int_{\Gamma_h} u_hv_h+\nath u_h\cdot\nath v_h \,ds_h =\int_{\Gamma_h}f_h v_h\, \,ds_h \quad \text{for all}~~v_h\in V_h^{^{\rm bulk}}. \label{FEM}
\end{equation}
 Here $f_h$ is a suitable approximation of $f$ on $\Gamma_h$. In the space  of \textit{traces} on $\Gamma_h$, $V_h^\Gamma:=\{v_h\in H^1(\Gamma_h)\,|\,v_h= v_h^{\rm bulk}|_{\Gamma_h},~v_h^{\rm bulk}\in V_h^{\rm bulk}\}$,  the solution of \eqref{FEM}  is unique. In other words, although in general there are multiple functions $u_h\in V_h^{\rm bulk}$ that satisfy \eqref{FEM}, the corresponding $u_h|_{\Gamma_h}$ is unique.
  Furthermore, under reasonable assumptions the following optimal error bound  holds:
\begin{equation} \label{eqqq1}
\|u^e-u_h\|_{L^2(\Gamma_h)}+h\|\nabla_{\Gamma_h}(u^e-u_h)\|_{L^2(\Gamma_h)} \le c\,h^2 \|u\|_{H^2(\Gamma)},
\end{equation}
where $u^e$ is a suitable extension of the solution to \eqref{LB} off the surface $\Gamma$ and $h$ denotes the mesh size of the \emph{outer} triangulation $\T_h$. The constant $c$ depends only on the shape regularity of $\T_h$  and is \textit{independent of how the surface $\Gamma_h$ cuts through the background mesh}. This robustness property is extremely important for extending the method to time-dependent  surfaces. It allows to keep the same background mesh while the surface evolves  through the bulk domain. One thus avoids unnecessary mesh fitting and mesh reconstruction.

A rigorous convergence analysis from which the result \eqref{eqqq1} follows will be given further on (section~\ref{sectanalysis}). Here we already mention two interesting properties of the induced surface triangulations which shed some light on why the method performs optimally for such shape \emph{ir}regular surface meshes as
 illustrated in Figure~\ref{fig:surfaceMesh}. These properties are the following: (i)~If the background triangulation $\T_h$ satisfies the minimum angle condition,  then the surface triangulation satisfies the \textit{maximum} angle condition ~\cite{olshanskii2012surface}; (ii)~Any element from  $\mathcal{F}_h$ shares at least one vertex with a full size shape regular triangle from $\mathcal{F}_h$~\cite{DemlowOlshanskii12}.

For the matrix-vector representation of the TraceFEM one uses the nodal basis of the bulk finite element space $V_h^{\rm bulk}$ rather than trying to construct a basis in $V_h^\Gamma$. This leads to singular or badly conditioned mass and stiffness matrices. In recent years stabilizations have been developed which are easy to implement and result in matrices with acceptable condition numbers. This linear algebra topic is treated is section~\ref{sectStiffness}.

In Part II of this article  we explain  how the ideas of the TraceFEM outlined above extend to the case of evolving surfaces.  For such  problems the method uses a \emph{space--time framework}, and the  trial and test finite element spaces consist of \emph{traces of standard volumetric elements on the space--time manifold}. This manifold results from the evolution of the surface.
The method stays essentially Eulerian as a surface is not tracked by a mesh.
Results of numerical tests  show that the method applies, without any modifications and without stability restrictions on mesh or time step sizes, to  PDEs on surfaces \textit{undergoing topological changes}.
We believe that this is a unique property of TraceFEM among the state-of-the-art surface finite element methods.

\subsection{Other surface Finite Element Methods}\label{secOther}

We briefly comment on other approaches known in the literature for solving PDEs on surfaces. A detailed overview of different finite element techniques for surface PDEs is given in \cite{DziukElliottAN}.
 The study of FEM for PDEs on general surfaces can be traced back to the paper  of Dziuk  \cite{Dziuk88}. In that paper, the Laplace--Beltrami equation is considered on a stationary surface $\Gamma$  approximated by a regular family $\{\Gamma_h\}$ of consistent  triangulations. It is assumed that all vertices in the triangulations lie on $\Gamma$. The finite element space then consists of scalar functions that are continuous on $\Gamma_h$ and linear on each triangle in the triangulation.  
The method is  extended from linear to higher order finite elements in \cite{Demlow09}. An adaptive finite element version of the method based on linear finite elements and suitable \textit{a posteriori} error estimators are treated in \cite{Demlow06}.
More recently,  Elliott and co-workers \cite{Dziuk07,DziukElliot2013a,elliott2015error} developed and analyzed an extension of the method of Dziuk for evolving surfaces. This surface finite element method is based on a Lagrangian tracking of the surface evolution. The surface $\Gamma(t)$ is approximated by an evolving triangulated surface $\Gamma_h(t)$. It is assumed that all vertices in the triangulation lie on $\Gamma(t)$ and a given bulk velocity field transports the vertices as material points (in the ALE variant of the method the tangential component of the transport velocity can be  modified to assure a better distribution of the vertices).
The finite element space then consists of scalar functions that are continuous on $\Gamma_h(t)$ and for each fixed $t$ they are linear on each triangle in the triangulation $\Gamma_h(t)$. Only recently a higher order evolving surface FEM has  been  studied in \cite{kovacs2016high}. 
If a surface undergoes strong deformations, topological changes, or  is defined implicitly, e.g., as the zero level of a level set function, then numerical methods based on such a   Lagrangian approach have certain disadvantages.

In order to avoid remeshing and make full use of the implicit definition of the surface as the zero of a  level set function, it was first proposed in \cite{bertalmio2001variational} to \emph{extend the  partial differential equation} from the surface to a set of positive Lebesgue measure in $\R^3$. The resulting PDE is then solved in one dimension higher but can be solved on a mesh that is unaligned to the surface.  Such an extension approach is studied in \cite{AS03,Greer,xu2006level,XuZh} for finite difference approximations, also for PDEs on moving surfaces. The extension approach can also be combined with finite element methods, see \cite{burger2009finite,DziukElliot2010,olshanskii2016narrow}. Another related method, which embeds a surface problem in a Cartesian bulk problem, is the closest point method of Ruuth and co-authors \cite{macdonald2009implicit,ruuth2008simple,petras2016pdes}. The method is based on using the closest point operator to extend the problem from the surface to a
small neighborhood of
the surface, where standard Cartesian finite differences are used to discretize differential operators. The surface PDE is then embedded and discretized in the neighborhood. Implementation requires the knowledge or calculation of the closest point on the surface for a given point in the neighborhood.
We are not aware of a finite element variant of the closest point method. Error analysis is also not known. The methods based on embedding a surface PDE in a bulk PDE are known to have certain issues such as the need of artificial  boundary conditions and difficulties in handling geometrical singularities, see, e.g., the discussion in \cite{Greer}.

The TraceFEM that we consider in this article, or very closely related methods, are also called \emph{CutFEM} in the literature, e.g. \cite{Alg1,burman2016cutb,burman2016full,burman2016cut}. Such CutFE techniques have originally been developed as unfitted finite element methods for interface problems, cf. the recent overview paper \cite{burman2015cutfem}. In such a method applied to a model Poisson interface problem one uses a standard finite element space on the whole domain and then ``cuts'' the functions from this space at the interface, which is equivalent to taking the trace of these functions on one of the subdomains (which are separated by the interface). In our TraceFEM one also uses a ``cut'' of  finite element functions from the bulk space, but now one cuts of the parts on both sides of the surface/interface and only keeps the part on the surface/interface. This explains why such trace techniques are also called \emph{Cut}-FEM.

\subsection{Structure of the article} The remainder of this article is divided into two parts. In the first part (sections~\ref{sectTraceFEM}-\ref{sectbulksurface}) we treat different aspects of the TraceFEM for stationary elliptic PDEs on a \emph{stationary} surface. As model problem we consider the Laplace--Beltrami equation \eqref{LB}. In section~\ref{sectTraceFEM} we give a detailed explanation of the TraceFEM and also consider a higher order isoparametric variant of the method. In section~\ref{sectStiffness} important aspects related to the matrix-vector representation of the discrete problem are treated. In particular several stabilization techniques are explained and compared. A discretization error analysis of TraceFEM is reviewed in section~\ref{sectanalysis}. Optimal (higher order) discretization error bounds are presented in that section. In section~\ref{sectconvdiff} we briefly treat a stabilized variant of TraceFEM that is suitable for convection dominated surface PDEs. A residual based \textit{
a posteriori}
error indicator for the TraceFEM is explained in section~\ref{sectadap}.  In the final section~\ref{sectbulksurface} of Part I  the Trace- or Cut-FEM  is applied for the discretization of a coupled bulk-interface mass transport model.

In the second part (sections~\ref{sec:evol}-\ref{sectVariants}) we treat different aspects of the TraceFEM for parabolic PDEs on an \emph{evolving} surface. In section~\ref{sec:evol}  well-posedness of a space--time weak formulation for a class of surface transport problems is studied.  A space--time variant of  TraceFEM is explained in section~\ref{sectSTTraceFEM} and some main results on stability and discretization errors for the method are treated in section~\ref{sectdiscranalysis}. A few recently developed variants of the space--time TraceFEM are briefly addressed in section~\ref{sectVariants}.

In view of the length of this article we decided not to present any results of numerical experiments. At the end of several sections we added remarks on numerical experiments (e.g. Remark~\ref{remnumexp1}) in which we refer to literature where results of numerical experiments for the methods that are treated are presented.
\\[3ex]
{\bf \Large Part I: Trace-FEM for stationary surfaces} \\[2ex]
In this part (sections~\ref{sectTraceFEM}-\ref{sectbulksurface}) we introduce the key ingredients of TraceFEM for elliptic and parabolic PDEs on \emph{stationary smooth surfaces}. The surface is denoted by $\Gamma$ and is assumed to be a smooth closed 2D surface, contained in a domain $\Omega \subset \Bbb{R}^3$.
We explain in more detail  how trace finite element spaces are used in a Galerkin method applied to the surface PDE. One important part of almost all numerical methods for surface PDEs is the numerical approximation of the surface. We address this topic, implementation aspects of the method, and properties of the stiffness matrix. Related to the latter topic we treat certain stabilization procedures for improving the conditioning of the stiffness matrix. We also discuss  an a-posteriori error indicator and an application of TraceFEM to coupled bulk-surface problems.

\section{Trace finite element method} \label{sectTraceFEM}
The trace finite element method applies to the variational formulation of a surface PDE.  We start with treating an elliptic problem and thus assume an $H^1(\Gamma)$ continuous and elliptic bilinear form $a(\cdot,\cdot)$, and for a given $f \in H^1(\Gamma)'$ we consider the following problem: find $u \in H^1(\Gamma)$ such that
\begin{equation} \label{contproblem}
 a(u,v)=f(v) \quad \text{for all}~~ v\in H^1(\Gamma).
\end{equation}
To simplify the presentation, we again restrict to the Laplace--Beltrami  model problem, i.e.,
\begin{equation} \label{eq:weak-LB}
  a(u,v):= \int_\Gamma\left( \unabla u \cdot \unabla v\, + u v\, \right)\rd s.
\end{equation}
We added the zero order term in this bilinear form to avoid the minor technical issue that for the problem with only the surface Laplacian one has to consider the bilinear form on the factor space $H^1(\Gamma)/\Bbb{R}$. The variational problem \eqref{contproblem} with the bilinear form defined in \eqref{eq:weak-LB} is well-posed. In section~\ref{sectconvdiff} we shall
consider another example, namely a surface convection-diffusion problem.

\subsection{Basic structure of TraceFEM} \label{sectbasic}
Let $\T_h$ be a  tetrahedral triangulation of the domain $\Omega \subset
\Bbb{R}^3$ that contains $\Gamma$.  This triangulation is assumed to be regular, consistent and stable~\cite{Braess};
it is the background mesh for the TraceFEM. On this background mesh,  $V_{h,j}$  denotes the standard space of $H^1$-conforming finite elements of degree $j \geq 1$,
\begin{equation} \label{defVm}
 V_{h,j}:=  \{\, v_h \in C(\Omega)\,|\,v_{h|T} \in \mathcal{P}_j~~\text{for all}~T \in \T_h\,\}.
\end{equation}
The nodal interpolation operator in $V_{h,j}$ is denoted by $I^j$.
We need an approximation $\Gamma_h$  of $\Gamma$. Possible constructions of $\Gamma_h$ and precise conditions that $\Gamma_h$ has to satisfy for the error analysis will be discussed later. For the definition of the method, it is sufficient to assume that
$\Gamma_h$ is a Lipschitz surface without boundary.
The active set of tetrahedra $\T_h^\Gamma \subset \T_h$ is defined by $\T_h^\Gamma=\{\, T \in  \T_h\,:\,{\rm meas}_2(\Gh \cap T) >0\, \}$. If $\Gh \cap T$ consists of a face $F$ of $T$, we include in $\T_h^\Gamma$ only one of the two tetrahedra which have this $F$ as their intersection. The domain formed by the tetrahedra from $\T_h^\Gamma$ is denoted further by $\omega_h$.  In the TraceFEM, only background degrees of freedom   corresponding to the tetrahedra from $\T_h^\Gamma$ contribute to algebraic systems.
Given a  bulk (background) FE space of degree $m$, $V_h^{\rm bulk}=V_{h,m}$, the corresponding \emph{trace space} is
\begin{equation} \label{defVG}
 V_{h}^\Gamma:= \{\, v_{h}|_{\Gh}\,:\, v_h \in V_h^{\rm bulk}\,\}.
\end{equation}
The trace space is a subspace of $H^1(\Gamma_h)$. On $H^1(\Gamma_h)$ one defines the finite element bilinear form,
\begin{equation*}
 a_h(u,v):=\int_{\Gh}( \unablah u \cdot \unablah v  + u v) \, \rd s_h.
\end{equation*}
The form is coercive on $H^1(\Gamma_h)$, i.e. $a_h(u_h,u_h) \geq \|u_h\|_{H^1(\Gamma_h)}^2$ holds. This guarantees that the TraceFEM has a unique solution in $V_{h}^\Gamma$. However, in TraceFEM formulations we prefer to use the background space $V_h^{\rm bulk}$ rather than $V_{h}^\Gamma$, cf. \eqref{FEM}, \eqref{discr1} and further examples in this paper. There are several reasons for this choice. First of all, in some versions of the method the \textit{volume} information from trace elements in $\omega_h$ is used; secondly, for implementation one uses nodal basis functions from $V_h^{\rm bulk}$ to  represent elements of $V_{h}^\Gamma$; thirdly, $V_{h}^\Gamma$ depends on the position of $\Gamma$, while  $V_h^{\rm bulk}$  does not;  and finally, the properties of $V_h^{\rm bulk}$ largely determine the properties of the method.
The trace space $V_{h}^\Gamma$ turns out to be  convenient for the analysis of the method.
 Thus, the basic form  of the TraceFEM for the discretization of \eqref{eq:weak-LB} is as follows: Find $u_h \in V_{h}^{\rm bulk}$ such that
\begin{equation} \label{discr1}
 a_h(u_h,v_h) = \int_{\Gh} f_h v_h \, \rd  s_h \quad \text{for all}~~v_h \in  V_{h}^{\rm bulk}.
\end{equation}
Here $f_h$ denotes  an approximation of the data $f$  on $\Gamma_h$. The construction of $f_h$ will be  discussed later, cf. Remark~\ref{daterror}.
Clearly, in \eqref{discr1} only the finite element functions $u_h,v_h \in V_{h}^{\rm bulk}$ play a role which have at least one  $T \in \T_h^\Gamma$ in their support.

\subsection{Surface approximation and isoparametric TraceFEM}   \label{sectGammaapprox}
One major ingredient in the TraceFEM (as in many other numerical methods for surface PDEs) is a construction of the surface approximation $\Gamma_h$. Several methods for numerical surface representation and approximation are known, cf.~\cite{DziukElliottAN}. In this paper we focus on the \emph{level set method} for surface representation. As it is well-known from the literature, the level set technique is a very popular method for surface representation, in particular  for handling evolving surfaces.

Assume that the surface $\Gamma$ is the zero level of a smooth level set function $\phi$, i.e.,
\begin{equation} \label{eq:Gamma}
\Gamma= \{\, x \in \Omega\,:\,\phi(x)=0\,\}.
\end{equation}
This level set function is not necessarily close to a signed distance function, but has the usual properties of a level set function:
$
\|\nabla \phi(x)\| \sim 1$, $\|D^2 \phi(x)\| \leq c$ for all $x$ in a neighborhood $U$ of $\Gamma$.
Assume that a finite element approximation $\phi_h\in V_{h,k}$  of the function $\phi$ is available.
If $\phi$ is sufficiently smooth, and one takes $\phi_h=I^k(\phi)$, then the estimate
\begin{equation} \label{eq:approx_phi}
\|\phi - \phi_h\|_{L^\infty(U)} + h \|\nabla(\phi - \phi_h)\|_{L^\infty (U)} \leq c h^{k+1}
\end{equation}
defines the accuracy of the geometry approximation by $\phi_h$. If $\phi$ is not known and $\phi_h$ is given, for example,
as the solution to the level set equation, then an estimate as in \eqref{eq:approx_phi} with some $k\geq 1$ is often assumed in
the error analysis of the TraceFEM.
In section~\ref{sectanalysis} we  explain how the accuracy of the geometry recovery influences the discretization error of the method. From the analysis we shall see that setting $m=k$ for the polynomial degree in  background FE space and the discrete level set function is the most natural choice.

The zero level of the finite element function $\phi_h$ (implicitly) characterizes an  interface approximation $\Gamma_h$:
\begin{equation} \label{eq:Gamma_h}
\Gamma_h= \{\, x \in \Omega\,:\,\phi_h(x)=0\,\}.
\end{equation}
With the exception of the linear case, $k=1$, the numerical integration over $\Gamma_h$ given implicitly in \eqref{eq:Gamma_h}  is a non-trivial problem. One approach to the numerical integration is based on an approximation of $\Gamma_h$ within each $T\in\T_h^\Gamma$ by elementary shapes. Sub-triangulations or
octree Cartesian meshes are commonly used for these purposes. On each elementary shape a standard quadrature rule is applied.  The approach is popular in combination with higher order XFEM, see, e.g., \cite{abedian2013performance,moumnassi2011finite,dreau2010studied},
and the level set method~\cite{min2007geometric,holdych2008quadrature}. Although numerically stable, the numerical integration based on sub-partitioning may significantly increase the computational complexity of a higher order finite element method.
Numerical integration over implicitly defined domains is a topic of current research, and in several recent papers~\cite{muller2013highly,saye2015high,fries2015higher,OlshSafin2,joulaian2016numerical} techniques  were developed that have optimal computational complexity. Among those, the moment--fitting method from ~\cite{muller2013highly} can be applied on 3D simplexes and, in the case of space--time methods, on 4D simplexes. The method, however,  is rather involved and the weights computed by the fitting procedure are not necessarily positive. As a computationally efficient alternative, we will treat below a higher order isoparametric TraceFEM, which  avoids the integration over a zero level of $\phi_h$.

The general framework of this paper, in particular the error analysis presented in section~\ref{sectanalysis}, provides an optimally accurate higher order method for PDEs on surfaces
both for the isoparametric approach and for approaches that make use  of a suitable integration procedure on implicitly defined domain as in \eqref{eq:Gamma_h}.

For piecewise linear polynomials  a computationally  efficient representation is straightforward. To exploit this property, we introduce the piecewise linear nodal interpolation of $\phi_h$, which   is denoted by $\hphi = \Ione \phi_h$. Obviously, we have $\hphi=\phi_h$ if $k=1$.  Furthermore, $\hphi(x_i)=\phi_h(x_i)$ at all vertices $x_i$ in the triangulation $\T_h$.
A lower order geometry approximation of the interface, which is very easy to determine,  is the zero level of this function:
$$\Gammalin := \{ x\in\Omega\mid \hphi(x) = 0\}.$$
In most papers on finite element methods for surface PDEs the surface approximation $\Gh=\Gammalin$ is used. This surface approximation is piecewise planar, consisting of triangles and quadrilaterals. The latter can be subdivided into triangles. Hence quadrature on  $\Gammalin$ can be reduced to quadrature on triangles, which is simple and computationally very  efficient.

Recently in \cite{lehrenfeld2015cmame} a computationally efficient \emph{higher order} surface approximation method has been introduced based on an isoparametric mapping. The approach from  \cite{lehrenfeld2015cmame} can be used to derive an efficient
higher order TraceFEM. We review the main steps below, while further technical details and analysis can be found in \cite{GLR}.
We need some further notation. All elements in the triangulation $\T_h$ which are cut by $\Gammalin$ are collected in the set $\T^{\Gamma_{\rm lin}}_h := \{T \in \mT\mid T \cap \Gammalin \neq \emptyset \}$. The corresponding domain is $\omega_h^{\rm lin} := \{ x \in T \mid T\in \T^{\Gamma_{\rm lin}}_h\}$.  We introduce a mapping $\Psi$ on $\omega_h^{\rm lin}$ with the property $\Psi(\Gammalin)=\Gamma$, which is defined as follows. Set $G:=\nabla \phi$, and define a  function $d: \omega_h^{\rm lin} \to \Bbb{R}$ such that $d(x)$ is the  smallest in absolute value number satisfying
\begin{equation} \label{cond1A}
  \phi(x + d(x) G(x))=\hphi(x)  \quad \text{for}~~x \in \omega_h^{\rm lin}.
\end{equation}
For $h$ sufficiently small the relation in \eqref{cond1A} defines a unique
 $d(x)$.
 Given the function
 $dG$
 we define:
\begin{equation} \label{psi1}
 \Psi(x):= x + d(x) G(x), \quad x \in \omega_h^{\rm lin}.
\end{equation}
From  $ \phi\big(\Psi(x)\big)=\hphi(x)$ it follows that $\phi\big(\Psi(x)\big)=0$ iff $\hphi(x)=0$, and thus $\Psi(\Gammalin)=\Gamma$ holds.
In general, e.g., if $\phi$ is not explicitly known,  the mapping $\Psi$ is not computable. We introduce an easy way to construct an accurate computable approximation of $\Psi$, which is based on $\phi_h$ rather than on $\phi$.

We define the polynomial extension
$\mathcal{E}_T: \mathcal{P}(T) \rightarrow \mathcal{P}(\mathbb{R}^3)$ so that for $v \in V_{h,k}$ we have $(\mathcal{E}_T v)|_T = v|_T,~T\in\T^{\Gamma_{\rm lin}}$.
For a  search direction $G_h \approx G $ one needs a sufficiently accurate approximation of $\nabla \phi$. One natural choice is
\[ G_h = \nabla \phi_h,
\]
but  there are also  other options.
Given $G_h$ we define  a function $d_h: \T^{\Gamma_{\rm lin}}_h \to \mathbb{R}$, $|d_h|\le\delta$, with $\delta > 0$ sufficiently small, as follows:
$d_h(x)$ is the  in absolute value  smallest  number such that
\begin{equation*} \label{eq:psihmap}
  \mathcal{E}_T \phi_h(x + d_h(x) G_h(x)) = \hphi(x), \quad \text{for}~~ x\in  T \in \T^{\Gamma_{\rm lin}}_h.
\end{equation*}
In the same spirit as above, corresponding to $d_h$ we define
\begin{equation*} \label{eq:psih}
  \Psi_h(x) := x + d_h(x) G_h(x), \quad \text{for}~~ x\in  T \in \T^{\Gamma_{\rm lin}}_h,
\end{equation*}
which is an approximation of the mapping $\Psi$ in \eqref{psi1}. For any fixed $x\in\T^{\Gamma_{\rm lin}}_h$ the value $\Psi_h(x)$ is easy to compute.
The mapping $\Psi_h$ may be discontinuous across faces and is not yet an isoparametric mapping.
To derive an isoparametric mapping, denoted by $\Theta_h$ below, one can use  a simple projection $P_h$ to map the transformation $\Psi_h$  into the continuous finite element space. For example,  one may define $P_h$ by averaging in a finite element node $x$, which requires only computing $P_h(x)$ for all elements sharing $x$. This results in
\begin{equation*} \label{defThetah} \Theta_h := P_h \Psi_h \in [V_{h,k}]^3.
\end{equation*}
 Based on this transformation one defines
\begin{equation} \label{Ghh}
\Gamma_h := \Theta_h(\Gammalin) = \{ x\in\Omega\,:\,\phi_h^{\rm lin} \big( \Theta_h^{-1}(x)\big) = 0\, \}.
\end{equation}
The finite element mapping $\Theta_h$ is completely characterized by its values at the finite element nodes. These values can be determined in a computationally very efficient way. From this it follows that for $\Gamma_h$ as in \eqref{Ghh} we have a computationally efficient representation. One can show that if \eqref{eq:approx_phi} holds then for both  $\Gamma_h$ defined in \eqref{eq:Gamma_h} or \eqref{Ghh} one gets (here and in the remainder the constant hidden in $\lesssim$ does not depend on how $\Gamma$ or $\Gamma_h$ intersects the triangulation $\T_h$):
\begin{equation} \label{resdist}
 {\rm dist}(\Gamma_h,\Gamma)=\max_{\bx\in \Gamma_h} {\rm dist}(\bx , \Gamma)  \lesssim h^{k+1}.
\end{equation}
For $\Gamma_h$ defined in  \eqref{Ghh}, however, we have a {computationally efficient higher order surface approximation} for all $k\ge1$. To allow an efficient quadrature in the TraceFEM on $\Gamma_h$, one also has to transform the background finite element spaces $V_{h,m}$ with the same transformation $\Theta_h$, as is standard in isoparametric finite element methods.  In this \emph{isoparametric} TraceFEM, we apply the local transformation $\Theta_h$ to the  space $V_{h,m}$:
\begin{equation}  \label{isoFEspace}
V_{h,\Theta}  = \{\,  v_h \circ \Theta_h^{-1} \mid v_h \in (V_{h,m})|_{\omega_h^{\rm lin}}\, \}= \{\,  (v_h \circ \Theta_h^{-1})|_{\Theta_h(\omega_h^{\rm lin})} \mid v_h \in V_{h,m} \, \}.
\end{equation}
The isoparametric TraceFEM discretization now reads, compare to \eqref{discr1}: Find $u_h \in V_{h,\Theta}$ such that
\begin{equation} \label{discriso}
\int_{\Gh} \unablah u_h \cdot \unablah v_h  + u_h v_h \, \rd s_h = \int_{\Gh} f_h v_h \, \rd  s_h \quad \text{for all}~~v_h \in  V_{h,\Theta},
\end{equation}
with $\Gamma_h := \Theta_h(\Gammalin)$. Again, the method in \eqref{discriso} can be reformulated in terms of the surface independent space $V_h^{\rm bulk}$, see \eqref{implemen}.

To balance the geometric and approximation errors, it is natural to take $m=k$, i.e., the same degree of polynomials is used in the approximation $\phi_h$ of $\phi$ and in the approximation $u_h$ of $u$.
The isoparametric TraceFEM is analyzed in \cite{GLR} and \emph{optimal order discretization error bounds} are derived.

\subsection{Implementation}
We comment on  an efficient implementation of the isoparametric TraceFEM. The integrals in \eqref{discriso} can be evaluated based on numerical integration rules with respect to $\Gammalin$ and the transformation $\Theta_h$. We illustrate this for the Laplacian part in the bilinear form. With $\tilde{u}_h = u_h \circ \Theta_h,~\tilde{v}_h = v_h \circ \Theta_h \in V_h^{\rm bulk}:=V_{h,m}$, there holds
\begin{equation} \label{implemen}
\int_{\Gamma_h} \nabla_{\Gamma_h}  u_h \cdot \nabla_{\Gamma_h}  v_h \, \rd{s_h} =
\int_{\Gammalin}  P_h (D \Theta_h)^{-T} \nabla \tilde{u}_h \cdot
\ P_h (D \Theta_h)^{-T} \nabla \tilde{v}_h\,  \mathcal{J}_\Gamma\, \rd{\tilde s_h},
\end{equation}
where $P_h = I - n_h n_h^T$ is the tangential projection, $n_h = N / \|N\|$ is the unit-normal on $\Gamma_h$, $N = (D \Theta_h)^{-T} \hat{n}_h$,  $\hat{n}_h = \nabla \hphi / \Vert\nabla \hphi\Vert$ is the normal with respect to $\Gammalin$, and $\mathcal{J}_\Gamma = \det(D \Phi_h) \Vert N \Vert$.
This means that one only needs  an accurate integration with respect to the low order geometry $\Gammalin$ and the explicitly available mesh transformation $\Theta_h \in [V_{h,k}]^3$. The terms occurring in  the integral on the right-hand side in \eqref{implemen}
are polynomial functions on each triangle element of $\Gammalin$.

We emphasize that taking $V_{h,\Theta}$ in place of  $V_{h,m}$ in \eqref{discriso} is important. For $V_{h,m}$ it is {not} clear how an efficient implementation can be realized. In that case one  needs to integrate over $\Gamma_T:=\Gammalin \cap T$ (derivatives of) the function $u_h \circ \Theta_h$, where $u_h$ is piecewise polynomial on $T \in \T_h$. Due to the transformation $\Theta_h \in [V_{h,k}]^3$ the function $u_h \circ \Theta_h$ has in  general not more than only Lipschitz smoothness on $\Gamma_T$. Hence an efficient and accurate quadrature becomes a difficult issue.
\begin{remark}[Numerical experiments] \label{remnumexp1} \rm Results of numerical experiments with the TraceFEM for $P_1$ finite elements ($m=1$) and a piecewise linear surface approximation ($k=1$) are given in \cite{ORG09}. Results for the higher order isoparametric TraceFEM are given in \cite{GLR}. In that paper, results of numerical  experiments with that method for $1 \leq k=m \leq 5$ are presented which confirm  the optimal high order convergence.

\end{remark}

\section{Matrix-vector representation and stabilizations} \label{sectStiffness}
The matrix-vector representation of the discrete problem in the  TraceFEM depends on the choice of a basis (or frame) in the trace finite element space. The most natural choice is to use the nodal basis of the outer space $V_{h,m}$ for representation of elements in the trace space $V_{h}^\Gamma$. This choice has been used in almost all papers on TraceFEM. It, however, has some  consequences. Firstly, in general the restrictions to $\Gamma_h$ of the outer nodal basis functions on $\T_h^\Gamma$ are {not} linear independent. Hence, these functions only form a frame and not a basis of the trace finite element space, and the corresponding mass matrix is singular. Often, however, the kernel of the mass matrix can be identified, and for $V_{h,1}$ elements it can be only one dimensional. Secondly, if one considers the scaled mass matrix on the space orthogonal to its kernel, the spectral condition number is typically {not} uniformly bounded with respect to $h$, but shows an $\mathcal{O}(h^{-2})$ growth. Clearly, this is different from the
standard uniform boundedness property of mass matrices in finite element discretizations. Thirdly, both for the mass and stiffness matrix there is a dependence of the condition numbers  on the location of the approximate interface $\Gamma_h$ with respect to the outer triangulation. In certain ``bad intersection cases'' the condition numbers can blow up.
A numerical illustration of some of these effects is given in  \cite{OR08}. Results of numerical experiments indicate that even for diagonally re-scaled (normalized) mass and stiffness matrices condition numbers become very large if higher order trace finite elements are used.

Clearly, the situation described above   concerning the conditioning of mass and stiffness matrices in the TraceFEM is not completely satisfactory, especially if a higher order method is of interest. In  recent literature  several \emph{stabilization methods} for TraceFEM have been introduced.
In these methods a stabilizing term is added to the bilinear form that results from the surface PDE (for example, the one in \eqref{discr1}). This stabilization term is designed to preserve the optimal discretization error bounds and at the same time ensure that the resulting mass and stiffness matrix have the full rank (apart from the kernel of Laplace--Beltrami operator) and have condition numbers $ch^{-2}$ with a constant $c$ that is \emph{in}dependent of how $\Gamma_h$ intersects the volume triangulation $\T_h$. Below we discuss the most important of these stabilization methods. All these methods are characterized by a bilinear form denoted by $s_h(\cdot,\cdot)$, and the stabilized discrete problem uses the same finite element space as the unstabilized one, but with a modified bilinear form
\begin{equation} \label{defAh}
  A_h(u,v):=a_h(u,v)+ s_h(u,v).
\end{equation}


\noindent\textbf{Ghost penalty stabilization.} \label{sec:ghostpen}
The ``ghost penalty'' stabilization is introduced in \cite{Burman2010} as a stabilization mechanism for unfitted finite element discretizations. In \cite{Alg1}, it is applied to a trace finite element discretization of the Laplace--Beltrami equation with piecewise linear finite elements ($m=k=1$).  For the ghost penalty stabilization, one considers the set of faces \emph{inside} $\omega_h$, $\mathcal{F}^\Gamma := \{ F =  \overline{T}_a \cap \overline{T}_b; T_a,T_b \in \mT^\Gamma, {\rm meas}_{2}(F) > 0\}$ and defines the face-based  bilinear form
$$
s_h(u_h,v_h) = \rho_s \sum_{F \in \mathcal{F}^\Gamma} \int_F \jump{\nabla u_h \cdot n_h} \jump{\nabla v_h \cdot n_h} \,\rd{s_h},
$$
with a stabilization parameter $\rho_s > 0$, $\rho_s \simeq 1$, $n_h$ is the normal to the face $F$ and $\jump{\cdot}$ denotes the jump of a function over the interface.
In \cite{Alg1}  it is shown that for piecewise linear finite elements, the stabilized problem results in a stiffness matrix (for the Laplace--Beltrami problem) with  a uniformly bounded condition number $\mathcal{O}(h^{-2})$.

Adding the jump of the derivatives on the element-faces changes, however, the sparsity pattern of the stiffness matrix. The face-based terms enlarge the discretization stencils.
To our knowledge, there is no higher order version of the ghost penalty method for surface PDEs which provides a uniform bound on the condition number.

\noindent\textbf{Full gradient surface stabilization}. \label{sec:fullgradsurf}
The ``full gradient'' stabilization is a method which does not rely on face-based terms and keeps the sparsity pattern intact. It was introduced in \cite{deckelnick2014unfitted,reusken2015}.  The bilinear form which describes this stabilization  is
\begin{equation} \label{fullgradient}
 s_h(u_h,v_h):= \int_{\Gamma_h} \nabla u_h \cdot n_h \, \nabla v_h \cdot n_h \, \rd{s_h},
\end{equation}
where $n_h$ denotes the normal to $\Gamma_h$.
Thus, we get $A_h(u_h,v_h)=\int_{\Gamma_h}( \nabla u_h \cdot \nabla v_h + u_h v_h)\, \rd{s_h}$, which explains the name of the method.
The stabilization is very easy to implement.

For the case of linear finite elements, it is shown in \cite{reusken2015} that one has a uniform condition number bound $\mathcal{O}(h^{-2})$ for diagonally re-scaled mass and stiffness matrices.
For the case of higher order TraceFEM ($m >1$), full gradient stabilization does not result in a uniform bound on the condition number, cf. \cite[Remark 6.5]{reusken2015}.

\noindent\textbf{Full gradient volume stabilization}. \label{sec:fullgradvol}
Another ``full gradient'' stabilization was introduced in \cite{burman2016full}. It uses the full gradient in the volume instead of only on the surface. The stabilization bilinear form is
$$
 s_h(u_h,v_h) = \rho_s \int_{\OGamma_{\Theta}} \nabla u_h \cdot \nabla v_h \rd{x},
$$
with a stabilization parameter $\rho_s > 0$, $\rho_s \simeq h$.
For $\Gamma_h= \Gammalin$ the domain $\OGamma_{\Theta}$ is just the union of tetrahedra intersected by $\Gamma_h$.
For application to an isoparametric TraceFEM as treated in section~\ref{sectGammaapprox} one should use the transformed domain $\OGamma_{\Theta}:=\Theta_h(\omega_h^{\rm lin})$.  This method is easy to implement as its bilinear form is provided by most finite element codes.
Using the analysis from \cite{burman2016full,GLR} it can be shown that a uniform condition number bound $\mathcal{O}(h^{-2})$ holds not only for linear finite elements but also for the higher order isoparametric TraceFEM. However, the stabilization is not ``sufficiently consistent'', in the sense that for the stabilized method one does {not} have the optimal order discretization bound for $m>1$.

\noindent\textbf{Normal derivative volume stabilization}.
In the lowest-order case $m=1$, all three stabilization methods discussed above result in a discretization which has a discretization error of optimal order \emph{and} a stiffness matrix with a uniform $\mathcal{O}(h^{-2})$ condition number bound. However, none of these methods  has both properties also for $m >1$.
We now discuss a recently introduced stabilization method \cite{burman2016cutc,GLR}, which does have both properties  for arbitrary $m \geq 1$. Its bilinear form is given by
\begin{equation}\label{eq:def-normal-gradient-volume-stab}
 s_h(u_h,v_h):= \rho_s \int_{\OGamma_\Theta} n_h \cdot \nabla u_h  \, n_h \cdot \nabla v_h \, \rd{x},
\end{equation}
with $\rho_s>0$ and $n_h$ the normal to $\Gamma_h$, which can easily be determined.
This is a natural variant of the full gradient stabilizations treated above. As in the full gradient surface stabilization only normal derivatives are added, but this time (as in the full gradient volume stabilization) in the volume $\OGamma_\Theta$.
The implementation of this stabilization term is fairly simple as it fits well into the structure of many finite element codes.
It can be shown, see \cite{GLR}, that using this stabilization in  the isoparametric TraceFEM  one obtains, for arbitrary $m=k \geq 1$, optimal order discretization bounds, and the resulting stiffness matrix has a spectral condition number $c h^{-2}$, where the constant $c$ does not depend on the position of the surface approximation $\Gamma_h$ in the triangulation $\T_h$. The bounds were proved in the $H^1$ norm, but do not foresee difficulties in showing the optimal error bounds in the $L^2$ norm as well. 
For these results to hold, one can take the scaling parameter $\rho_s$ from the following (very large) parameter range:
\begin{equation}\label{eq:as-rho}
  h \lesssim \rho_s \lesssim h^{-1}.
\end{equation}
Results of numerical experiments which illustrate the dependence of discretization errors and condition numbers on $\rho_s$ are given in \cite{GLR}. 
\section{Discretization error analysis}  \label{sectanalysis}
In this section we present a general framework in which both optimal order discretization bounds can be established and the conditioning of the resulting stiffness matrix can be analyzed. Our exposition follows the papers \cite{reusken2015,GLR}. In this framework we  need certain ingredients such as approximation error bounds for the finite element spaces, consistency estimates for the geometric error and certain fundamental properties of the stabilization. The required results are scattered in the literature and can be found in many papers, some of which we refer to below.

For the discretization we need an approximation $\Gamma_h$ of $\Gamma$. We do not specify a particular construction for $\Gamma_h$ at this point, but only assume certain properties introduced in section~\ref{sectprelim} below. This $\Gamma_h$  may, for example, be constructed via a mapping $\Theta_h$ as section~\ref{sectGammaapprox}, i.e., $\Gamma_h= \Theta_h(\Gammalin)$ or it may be characterized as the zero level of a (higher order) level set function $\phi_h$, cf. \eqref{eq:Gamma_h}. In the latter case, to perform quadrature on $\Gamma_h$  one does not use any special transformation but applies a ``direct'' procedure, e.g., a subpartition technique or the moment-fitting method. This difference (direct access to $\Gamma_h$ or access via $\Theta_h$) has to be taken into account in the definition of the trace spaces. We want to present an analysis which covers both cases and therefore
 we introduce a local bijective mapping $\Phi_h$, which is either $\Phi_h=\Theta_h$ ($\Gamma_h$ is accessed via transformation $\Theta_h$), cf.~\eqref{isoFEspace},  or $\Phi_h={\rm id}$ (direct access to $\Gamma_h$) and define
\begin{equation*}  \label{isoFEspaceA}
 V_{h,\Phi}  := \{\,  (v_h \circ \Phi_h^{-1})_{|\OGamma_\Phi} \mid v_h \in V_{h,m}\, \},
 \end{equation*}
where $\OGamma_\Phi$ is the domain formed by all (transformed) tetrahedra that are intersected by $\Gamma_h$.

We consider the bilinear form $A_h$ from \eqref{defAh}
with a general symmetric positive semidefinite bilinear form $s_h(\cdot,\cdot)$. Examples of $s_h(\cdot,\cdot)$ are $s_h\equiv 0$ (no stabilization) and the ones discussed in section~\ref{sectStiffness}. 
The discrete problem is as follows:
Find $u_h \in \Vk$ such that
\begin{equation} \label{eq:varform}
 A_h(u_h,v_h) = \int_{\Gamma_h} f_h v_h \, \rd{s_h} \quad \text{ for all } v_h \in \Vk.
\end{equation}
In the sections below we present a general framework for discretization error analysis of this method and outline main
results. Furthermore the conditioning of the resulting stiffness matrix is studied.
\subsection{Preliminaries} \label{sectprelim}
We collect some notation and results that we need in the error analysis. We always assume that $\Gamma$ is sufficiently smooth without specifying the regularity of $\Gamma$. The signed distance function to $\Gamma$ is denoted by $d$, with $d$ negative in  the interior of $\Gamma$.
On
$U_\delta: = \{\, x \in  \R^3\,:\, |d(x)| < \delta\, \}$,
 with $\delta >0$ sufficiently small, we define
\begin{align}
 & n(x)=\nabla d(x),  ~H(x)=D^2 d(x), ~~P(x)=I-n(x)n(x)^T, \\
 & p(x)= x- d(x)n(x), ~~ v^e(x)=v(p(x))~~\text{for $v$ defined on $\Gamma$}. \label{notat}
\end{align}
The eigenvalues of $H(x)$ are denoted by $\kappa_1(x),\kappa_2(x)$ and 0. Note that $v^e$ is simply the constant extension of $v$ (given on $\Gamma$) along the normals $n$. The tangential derivative can be written as $\unabla g(x)= P(x) \nabla g(x)$ for
$x \in \Gamma$. We assume $\delta_0 >0$ to be sufficiently small such that on $U_{\delta_0}$ the decomposition
\[
x= p(x)+ d(x) n(x)
\]
is unique for all $x \in U_{\delta_0}$. In the remainder we only consider $U_\delta$ with $0<  \delta \leq \delta_0$.
In the analysis we use basic transformation formulas  (see,  e.g.,\cite{Demlow06}). For example:
\begin{align}
\nabla u^e (x) & = (I -  d(x) H(x))\unabla u(p(x)) \quad \text{a.e on}~~U_{\delta_0}, ~u \in H^1(\Gamma).
\label{nabla1}
\end{align}
Sobolev norms of $u^e$ on $U_\delta$ are related to corresponding norms on $\Gamma$. Such results are known in the literature, e.g. \cite{Dziuk88,Demlow06}. We will use the following result:
\begin{lemma} \label{lemma1}  For $\delta >0$ sufficiently small the following holds.
For all $u \in H^m(\Gamma)$:
\begin{align}
\|D^\mu u^e\|_{L^2(U_\delta)} & \le c \sqrt{\delta}\|u\|_{H^m(\Gamma)}, \quad |\mu|=m \geq 0,\label{H2}
\end{align}
with a constant $c$ independent of $\delta$ and $u$.
\end{lemma}

\subsection{Assumptions on surface approximation $\Gamma_h$}
We already discussed some properties of $\Gamma_h$ defined in \eqref{eq:Gamma_h} and \eqref{Ghh}. In this section we formulate more precisely the properties of a generic discrete surface $\Gamma_h$ required  in the  error analysis.

The surface approximation $\Gamma_h$  is assumed to be a closed connected Lipschitz manifold.  It can be partitioned as follows:
\[
  \Gamma_h = \bigcup\limits_{T \in \T_h^\Gamma} \Gamma_T, ~~ \Gamma_T:= \Gamma_h \cap T.
\]
The unit normal (pointing outward from the interior of $\Gh$) is denoted by $n_h(x)$, and is defined a.e. on $\Gh$.
The first assumption is rather mild.
\begin{assumption} \label{ass1} \rm (A1) We assume that there is a constant $c_0$ independent of $h$ such that for  the  domain $\omega_h$ we have
\begin{equation} \label{cond1}
  \omega_h \subset U_{\delta}, \quad \text{with}~~ \delta =c_0 h \leq \delta_0.
\end{equation}
(A2)  We assume that for each $T \in \T_h^\Gamma$ the local surface section  $\Gamma_T$ consists of simply connected parts $\Gamma_T^{(i)}$, $i=1,\ldots p$,
and $\|n_h(x)-n_h(y)\| \leq c_1 h$ holds for $x,y \in \Gamma_T^{(i)},~i=1,\ldots p$. The number $p$ and constant $c_1$ are uniformly bounded w.r.t. $h$ and $T\in \T_h$.
\end{assumption}
\begin{remark} \label{remass1} \rm
 The condition (A1) essentially means that ${\rm dist}(\Gamma_h,\Gamma) \leq c_0 h$ holds, which is a very mild condition on the accuracy of $\Gh$ as an approximation of $\Gamma$. The condition ensures that the local triangulation $\T_h^\Gamma$ has sufficient resolution for  representing the  surface $\Gamma$ approximately. The condition (A2) allows multiple intersections (namely $p$) of $\Gamma_h$ with one tetrahedron $T \in \T_h^\Gamma$, and requires a (mild) control on the normals of the surface approximation. We discuss three situations in which Assumption~\ref{ass1} is satisfied. For the case $\Gamma_h =\Gamma$ and with $h$ sufficiently small both conditions in Assumption~\ref{ass1} hold.  If  $\Gamma_h$ is a shape-regular triangulation, consisting of triangles with diameter $\mathcal{O}(h)$ and vertices on $\Gamma$, then for $h$ sufficiently small both conditions are also  satisfied. Finally, consider the case in which $\Gamma$ is the zero level of a smooth level set function $\phi$, and $\phi_h$ is a
finite element approximation of $\phi$ on the triangulation $\T_h$. If $\phi_h$ satisfies \eqref{eq:approx_phi} with $k=1$ and  $\Gamma_h$
 is the zero level of $\phi_h$, see \eqref{eq:Gamma_h},  then the conditions (A1)--(A2) are satisfied, provided $h$ is sufficiently small.
\end{remark}
\medskip

For the analysis of the \emph{approximation error} in the TraceFEM one only needs Assumption~\ref{ass1}. For this  analysis, the following result is  important.
\begin{lemma} \label{LemHansbo}
 Let {\rm (A2)} in Assumption~\ref{ass1} be satisfied. There exist  constants $c$, $h_0 >0$, independent of how $\Gamma_h$ intersects $\T_h^\Gamma$, and with $c$ independent of $h$, such that for $h \leq h_0$ the following holds. For all $T \in \T_h^\Gamma$ and all $v \in H^1(T)$:
\begin{equation} \label{inHansbo}
 \|v\|_{L^2(\Gamma_T)}^2 \leq c \big( h_T^{-1} \|v\|_{L^2(T)}^2 + h_T \|\nabla v\|_{L^2(T)}^2\big),
\end{equation}
with $h_T:={\rm diam}(T)$.
\end{lemma}

The inequality \eqref{inHansbo} was introduced in \cite{Hansbo02}, where one also finds a proof under a somewhat more restrictive assumption. Under various (similar) assumptions, a proof of the estimate in \eqref{inHansbo} or of very closely related ones is found in \cite{Hansbo04,reusken2015,ChernOlsh1}.
For deriving higher order \emph{consistency bounds for the geometric error} we need a further more restrictive assumption introduced below.
\begin{assumption} \label{ass2}
\rm Assume that $\Gamma_h \subset U_{\delta_0} $ and that the projection $p:\,\Gamma_h \to \Gamma$ is a bijection.  We assume that the following holds, for a $k \geq 1$:
\begin{align} \|d\|_{L^\infty(\Gh)} & \leq c h^{k+1}, \label{condA} \\
 \|n -n_h\|_{L^\infty(\Gh)} & \leq c h^{k}.  \label{condB}
\end{align}
\end{assumption}

Clearly, if $\Gamma_h=\Gamma$ there is no geometric error, i.e.  \eqref{condA}--\eqref{condB} are fulfilled with $k=\infty$.
If $\Gamma_h$ is defined as  in \eqref{eq:Gamma_h}, and \eqref{eq:approx_phi} holds, then the conditions \eqref{condA}--\eqref{condB} are satisfied with the same $k$ as in \eqref{eq:approx_phi}. In \cite{Demlow09} another method for constructing polynomial approximations to $\Gamma$ is presented that satisfies the conditions \eqref{condA}--\eqref{condB} (cf. Proposition 2.3 in \cite{Demlow09}). In that method the exact distance function to $\Gamma$ is needed. Another method, which does not need information about the exact distance function, is introduced in \cite{Grande2014}.  A further  alternative is the method presented in section~\ref{sectGammaapprox}, for which it also can be shown that the conditions \eqref{condA}--\eqref{condB} are satisfied.

The surface measures  on $\Gamma$ and $\Gamma_h$ are related through the identity
\begin{equation} \label{surfmeasuer}
 \mu_h ds_h(x) = ds(p(x)), \quad \text{for}~~ x \in \Gamma_h.
\end{equation}
If Assumption~\ref{ass2} is satisfied the estimate
\begin{equation} \label{estmu}
 \|1- \mu_h\|_{\infty,\Gamma_h} \lesssim h^{k+1}
\end{equation}
holds, cf. \cite{Demlow06,reusken2015}.

\subsection{Strang Lemma} \label{sectStrang}
In the error analysis of the method we also need the following larger (infinite dimensional) space:
$$
 \Vregh := \{ v \in H^1(\OGamma_{\Phi})\,:\, v|_{\Gamma_h}  \in H^1(\Gamma_h)\} \supset \Vk,
$$
on which the bilinear form $A_h(\cdot,\cdot)$ is  well-defined. The natural (semi-)norms that we use in the analysis are
\begin{equation}\label{eq:def-h-norm}
  \enormh{u}^2 :=\|u\|_a^2 + s_h(u,u), \quad \|u\|_a^2:=a_h(u,u),  \quad u \in \Vregh.
\end{equation}

\begin{remark} \rm
On $V_{h,\Phi}^\Gamma$ the semi-norm $\|\cdot\|_a$ defines a norm.  Therefore, for a solution $u_h\in\Vk$ of the discrete problem \eqref{eq:varform}, the trace $u_h|_{\Gamma_h}\in V_{h,\Phi}^\Gamma$ is unique. The uniqueness of $u_h\in\Vk$ depends on the stabilization term and will be addressed in Remark~\ref{uniquesol} below.
\end{remark}

The following Strang Lemma is the basis for the error analysis. This basic result is used in almost all error analysis of TraceFEM and can be found in many papers. Its proof is elementary.
\begin{lemma} \label{Strang} Let $u \in H^1_0(\Gamma)$ be the unique solutions of \eqref{eq:weak-LB} with the extension $u^e \in \Vregh$ and let $u_h \in \Vk$ be a solution of \eqref{eq:varform}. Then we have the discretization error bound
\begin{equation} \label{Strangbound}
 \enormh{u^e-u_h} \leq 2 \min_{v_h \in \Vk} \enormh{u^e - v_h} + \sup_{w_h \in  \Vk} \frac{|A_h(u^e,w_h)-\int_{\Gamma_h} f_h w_h \, \rd{s_h}|}{\enormh{w_h}}.
\end{equation}
\end{lemma}
\subsection{Approximation error bounds} \label{sectapprox} In the approximation error analysis one derives bounds for the first term on the right-hand side in \eqref{Strangbound}. Concerning the quality of the approximation $\Gamma_h \sim \Gamma$ one needs only  Assumption~\ref{ass1}. Given the mapping $\Phi_h$, we define the (isoparametric) interpolation $I_\Phi^m:\, C(\OGamma_\Phi) \to \Vk$ given by $(I_\Phi^m v)\circ \Phi_h= I^m(v\circ \Phi_h)$. We assume that the following optimal interpolation error bound holds for
all $0 \leq l \leq m+1$:
\begin{equation} \label{interpol}
 \|v - I_\Phi^m v\|_{H^l(\Phi_h(T))} \lesssim  h^{m+1-l} \|v\|_{H^{m+1}(\Phi_h(T))} \quad \text{for all}~~v \in H^{m+1}(\Phi_h(T)), ~~T \in \mT.
\end{equation}
Note that this is an interpolation estimate on the outer shape regular triangulation $\T_h$.
For $\Phi_h={\rm id}$ this interpolation bound holds due to standard finite element theory. For $\Phi_h=\Theta_h$ the bound follows from the theory on isoparametric finite elements, cf.~\cite{Lenoir86,GLR}.
Combining this with the trace estimate of Lemma~\ref{LemHansbo} and the estimate $\|v^e\|_{H^{m+1}(\OGamma_{\Phi})} \lesssim h^\frac12 \|v\|_{H^{m+1}(\Gamma)}$ for all $v \in H^{m+1}(\Gamma)$, which follows from \eqref{H2}, we obtain the result in the following lemma.
\begin{lemma} \label{intertrace} For the  space $\Vk$ we have the approximation error estimate
\begin{equation}
 \begin{split}
 & \min_{v_h \in \Vk} \big( \|v^e-v_h\|_{L^2(\Gamma_h)} + h \|\nabla(v^e-v_h)\|_{L^2(\Gamma_h)} \big) \\
 & \leq  \|v^e-I_\Phi^m v^e \|_{L^2(\Gamma_h)} + h \|\nabla(v^e- I_\Phi^m v^e)\|_{L^2(\Gamma_h)}
 \lesssim h^{m+1} \|v\|_{H^{m+1}(\Gamma)}
\end{split}
\end{equation}
for all $v \in H^{m+1}(\Gamma)$. Here $v^e$ denotes the constant extension of $v$ in normal direction.
\end{lemma}
\medskip

Finally we obtain an optimal order bound for the  approximation  term in the Strang Lemma by combining the result in the previous lemma with an appropriate assumption on the stabilization bilinear form.
\begin{lemma} Assume that the stabilization satisfies
\begin{equation} \label{conds1}
 s_h(w,w) \lesssim h^{-3}\|w\|_{L^2(\OGamma_\Phi)}^2 + h^{-1} \|\nabla w\|_{L^2(\OGamma_\Phi)}^2 \quad \text{for all} ~~w \in \Vregh.
\end{equation}
Then it holds
\[
  \min_{v_h \in \Vk} \enormh{u^e - v_h} \lesssim h^m \|u\|_{H^{m+1}(\Gamma)} \quad \text{for all}~~u \in H^{m+1}(\Gamma).
\]
\end{lemma}
\begin{proof}
 Take $u \in H^{m+1}(\Gamma)$ and $v_h:= I_\Phi^m u^e$. From Lemma~\ref{intertrace} we get $\|u^e-v_h\|_a \lesssim h^m \|u\|_{H^{m+1}(\Gamma)} $.
 From the assumption \eqref{conds1} combined with the results in \eqref{interpol}
 we get $s_h(u^e-v_h,u^e-v_h)^\frac12 \lesssim h^m \|u\|_{H^{m+1}(\Gamma)}$, which completes the proof.
\end{proof}

\subsection{Consistency error bounds} \label{sectconsist}
In the consistency analysis, the geometric error is treated, and for obtaining optimal order bounds we need  Assumption~\ref{ass2}. One has to quantify the accuracy of the data extension $f_h$.
With $\mu_h$ from \eqref{surfmeasuer} we set
$
  \delta_f:=f_h- \mu_h f^e~\text{on}~\Gamma_h.
$

 \begin{lemma} \label{lem:conserr}
Let $u \in H^1(\Gamma)$ be the solution of \eqref{eq:weak-LB}. Assume that the data error satisfies $\|\delta_f\|_{L^2(\Gamma_h)} \lesssim h^{k+1} \|f\|_{L^2(\Gamma)}$ and the stabilization satisfies
\begin{equation} \label{conds2}
 \sup_{w_h \in \Vk} \frac{s_h(u^e,w_h)}{\enormh{w_h}} \lesssim h^{l}\|f\|_{L^2(\Gamma)} \quad \text{for some}~l\ge0. 
\end{equation}
 Then the following holds:
\[ \sup_{w_h \in \Vk} \frac{|A_h(u^e,w_h)-\int_{\Gamma_h} f_h w_h \, \rd{s_h}|}{\enormh{w_h}} \lesssim (h^{l}+h^{k+1})\|f\|_{L^2(\Gamma)}.
\]
\end{lemma}
\begin{proof}
 We use the splitting
\[
   |A_h(u^e,w_h)-\int_{\Gamma_h} f_h w_h \, \rd{s_h}| \leq |a_h(u^e,w_h)-\int_{\Gamma_h} f_h w_h \rd{s_h}| + | s_h(u^e,w_h)|.
\]
The first term has been treated in many papers. A rather general result, in which one  needs Assumption~\ref{ass2} and the bound on the data error, is given in \cite{reusken2015}, Lemma 5.5.
The analysis yields
\[
  \sup_{w_h \in  \Vk} \frac{|a_h(u^e,w_h)-\int_{\Gamma_h} f_h w_h \, \rd{s_h}|}{\enormh{w_h}} \lesssim h^{k+1}\|f\|_{L^2(\Gamma)}.
\]
We use assumption \eqref{conds2} to bound the second term.
\end{proof}
\begin{remark} \rm
\label{daterror}
 We comment on the data error $\|\delta_f\|_{L^2(\Gamma_h)}$. If we assume $f$ to be defined in a  neighborhood $U_{\delta_0}$ of $\Gamma$  one can then use
\begin{equation} \label{extens}
 f_h(x)= f(x).
\end{equation}
Using Assumption~\ref{ass2},  \eqref{estmu} and a Taylor expansion we get  $\|f- \mu_h f^e\|_{L^2(\Gamma_h)} \leq c h^{k+1} \|f\|_{H^{1,\infty}(U_{\delta_0})}$. Hence, a data error bound $\|\delta_f\|_{L^2(\Gamma_h)} \leq \hat c h^{k+1}\|f\|_{L^2(\Gamma)}$ holds with $\hat c=\hat c(f)=c \|f\|_{H^{1,\infty}(U_{\delta_0})} \|f\|_{L^2(\Gamma)}^{-1}$  and a constant $c$ independent of $f$. Thus, in problems with smooth data, $ f \in H^{1,\infty}(U_{\delta_0})$,   the extension defined in  \eqref{extens} satisfies the  condition on  the data error in Lemma~\ref{lem:conserr}.
\end{remark}

\subsection{TraceFEM error bound and conditions on $s_h(\cdot,\cdot)$}
As a corollary of the results in the sections~\ref{sectStrang}--\ref{sectconsist} we obtain the following main theorem on the error of TraceFEM.
\begin{theorem} \label{mainthm}
Let $u \in H^{m+1}(\Gamma)$ be the solution of \eqref{eq:weak-LB} and $u_h \in \Vk$ a solution of \eqref{eq:varform}. Let the Assumptions~\ref{ass1} and \ref{ass2} be satisfied and assume that  the data error bound $\|\delta_f\|_{L^2(\Gamma_h)} \lesssim h^{k+1} \|f\|_{L^2(\Gamma)}$ holds. Furthermore,  the stabilization should satisfy the conditions \eqref{conds1}, \eqref{conds2}. Then the following {\rm a priori} error estimate holds:
\begin{equation}
  \enormh{u^e - u_h }  \lesssim h^m \Vert u \Vert_{H^{m+1}(\Gamma)} + (h^{l}+h^{k+1}) \Vert f \Vert_{L^2(\Gamma)},
\end{equation}
where $m$ is the polynomial degree of the background FE space, $k+1$ is the order of surface approximation from  Assumption~\ref{ass2}, see also \eqref{eq:approx_phi}, and $l$ is the degree of consistency of the stabilization term, see \eqref{conds2}.
\end{theorem}

\begin{remark} \rm Optimal order error bounds in the $L^2$-norm are also known in the literature for the stabilized TraceFEM and for the original variant without stabilization with $m=k=1$, \cite{Alg1,ORG09}. For the higher order case with  $\Phi_h=\text{id}$ and $s_h \equiv 0$,
the  optimal order estimate
\[
\|u^e - u_h\|_{L^2(\Gamma_h)} \lesssim h^{m+1} \Vert u \Vert_{H^{m+1}(\Gamma)} + h^{k+1}\Vert f \Vert_{L^2(\Gamma)}
\]
is derived in \cite{reusken2015}. We expect that the analysis can be extended to the isoparametric  variant
of the TraceFEM, but this has not been done, yet.
 \end{remark}
\medskip

The conditions \eqref{conds1} and \eqref{conds2} on the stabilization allow an optimal order discretization error  bound. Clearly these conditions are satisfied for $s_h(\cdot,\cdot) \equiv 0$. Below we introduce a third condition, which has a different nature. This condition allows a uniform $\mathcal{O}(h^{-2})$ condition number bound for the stiffness matrix. The latter matrix is the representation of $A_h(\cdot,\cdot)$ in terms of standard nodal basis functions on the background mesh $\T_h^\Gamma$.
The following theorem is proved in \cite{GLR}.
\begin{theorem}\label{coro:condition-number}
Assume that the stabilization satisfies \eqref{conds1} and that
\begin{equation}
a_h(u_h,u_h) + s_h(u_h,u_h) \gtrsim  h^{-1} \| u_h\|_{L^2(\Omega^\Gamma_\Phi)}^2 \quad \text{for all}~~u_h \in \Vk. \label{conds3}
\end{equation}
Then, the spectral condition number of the stiffness matrix corresponding to $A_h(\cdot,\cdot)$ is bounded by $c h^{-2}$, with a constant $c$ independent of $h$ and of the location of $\Gamma_h$ in the triangulation $\T_h$.
\end{theorem}

\begin{remark} \label{uniquesol} \rm
From Theorem~\ref{coro:condition-number} it follows that if the stabilization satisfies \eqref{conds1} and  \eqref{conds3} then the stiffness matrix has  full rank and thus the discrete problem \eqref{eq:varform} has a unique solution.
\end{remark}
\medskip

We summarize the assumptions on the stabilization term $s_h$ used to derive Theorem \ref{mainthm} (optimal discretization error bound) and Theorem \ref{coro:condition-number} (condition number bound):
\begin{subequations} \label{conds}
\begin{align}
  s_h(w,w) & \lesssim h^{-3}\|w\|_{L^2(\OGamma_\Phi)}^2 + h^{-1} \|\nabla w\|_{L^2(\OGamma_\Phi)}^2 \quad \text{for all}~w \in \Vregh, \label{conds1A} \\
\sup_{w_h \in  \Vk} \frac{s_h(u^e,w_h)}{\enormh{w_h}}  & \lesssim h^{l}\|f\|_{L^2(\Gamma)}, \quad \text{with }l\ge0, 
\label{conds2A} \\
a_h(u_h,u_h)& + s_h(u_h,u_h)  \gtrsim  h^{-1} \| u_h\|_{L^2(\OGamma_\Phi)}^2 \quad \text{for all}~~u_h \in \Vk. \label{conds3A}
\end{align}
\end{subequations}
In \cite{GLR} these conditions are studied for various stabilizations. It is shown that for $m=k=1$  all four stabilization methods discussed in section~\ref{sectStiffness} satisfy these three conditions with $l=2$. Hence, these methods lead to optimal first order discretization error bounds and uniform $\mathcal{O}(h^{-2})$ condition number bounds. For higher order elements and geometry recovery, $m=k \geq 2$, however, only the normal derivative volume stabilization satisfies these conditions with $l=k+1$.\\
%
%
\section{Stabilized TraceFEM for surface convection--diffusion equations} \label{sectconvdiff}
  Assume we are given a smooth vector field $\bw$ everywhere tangential to the surface $\Gamma$. Another model problem of interest is the surface advection-diffusion equation,
\begin{equation}
u_t+ \mathbf{w}\cdot\nabla_{\Gamma} u+ (\Div_\Gamma\bw)u -\varepsilon\Delta_{\Gamma} u=0
\quad\text{on}~~\Gamma.
\label{e:2.2}
\end{equation}
In section~\ref{sec:evol} we shall consider equations modelling the conservation of a scalar quantity $u$ with a diffusive flux on an evolving surface $\Gamma(t)$, which is passively advected by a velocity field $\bw$. The equation  \eqref{e:2.2} represents a particular case of this problem, namely  when $\bw\cdot \bn=0$ holds, meaning that the surface is \emph{stationary}. A finite difference  approximation of $u_t$  results in the elliptic surface PDE:
\begin{equation}
  -\varepsilon\Delta_{\Gamma} u+\mathbf{w}\cdot\nabla_{\Gamma} u + (c+\Div_\Gamma\bw)\,u =f\quad\text{on}~~\Gamma. \label{problem}
\end{equation}
We make the following regularity  assumptions on the data:  $f\in L^2(\Gamma)$, $c=c(\bx)\in L^\infty(\Gamma)$, $\bw\in H^{1,\infty}(\Gamma)^3$.
Integration by parts over $\Gamma$ and using $\mathbf{w}\cdot\bn=0$ leads us to the weak formulation \eqref{contproblem} with
\begin{equation*}
a(u,v):=\int_{\Gamma}(\eps\nabla_{\Gamma} u\cdot\nabla_{\Gamma} v  - (\bw\cdot\nat v) u
+c\, uv)\, {\rm ds}.
\end{equation*}
Note that for $c=0$ the source term in \eqref{problem} should satisfy the zero mean constraint $\int_\Gamma f\,{\rm ds}=0$.
For  well-posedness of the variational formulation in $H^1(\Gamma)$ it is sufficient to assume
\begin{equation}\label{A1}
c+\frac12\Div_\Gamma\bw\ge c_0>0~~\text{on}~~\Gamma.
 \end{equation}
For  given extensions $\bw_h$, $c_h$, and $f_h$ off the surface to a suitable neighborhood, the formulation of the TraceFEM
or isoparametric TraceFEM is similar to the one for the Laplace--Beltrami equation. However, as in the usual   Galerkin finite element method for  convection--diffusion equations on a planar domain, for the case of strongly dominating convection the method would be prone to instabilities if the mesh is not sufficiently fine.
To handle the case of dominating convection, a SUPG type stabilized TraceFEM was introduced and analyzed in \cite{ORX}.  The stabilized formulation reads: Find $u_h\in V_h^{\rm bulk}$ such that
\begin{multline}
 \int_{\Gamma_h}(\eps\nath u_h\cdot\nath v_h\, - (\bw_h\cdot\nath v_h) u_h+c_h\, u_hv_h)\, {\rm ds}_h \\
  +\sum_{T\in\T_h^\Gamma}\delta_T\int_{\Gamma_T}\big(-\eps\Delta_{\Gamma_h}u_h + \bw_h\cdot\nath u_h + (c_h+\Div_{\Gamma_h}\bw_h) \, u_h\big)\,\bw_h\cdot\nath v_h\, {\rm ds}_h\\ =\int_{\Gamma_h}f_h v_h\, {\rm ds}_h + \sum_{T\in\T_h^\Gamma}\delta_T\int_{\Gamma_T}f_h(\bw_h\cdot\nath v_h)\, {\rm ds}_h\quad \forall~ v_h\in V_h^{\rm bulk}. \label{FEM_SUPG}
\end{multline}
The analysis of \eqref{FEM_SUPG} was carried out in \cite{ORX} for the lowest order method, $k=m=1$.
Both analysis and numerical experiments in \cite{ORX} and \cite{chernyshenko2015adaptive} revealed that the properties
of the stabilized formulation \eqref{FEM_SUPG} remarkably resemble those of the well-studied SUPG method for planar case.
In particular, the stabilization parameters  $\delta_T$ may be designed following the standard arguments, see, e.g., \cite{TobiskaBook}, based on mesh Peclet numbers for background tetrahedra and independent of how $\Gamma_h$ cuts through the mesh.
One particular choice resulting from the analysis is
\begin{equation}
 \widetilde{\delta_T}=
\left\{
\begin{aligned}
&\frac{\delta_0 h_{T}}{\|\mathbf{w}_h\|_{L^\infty(\Gamma_T)}} &&\quad \hbox{ if } \mathsf{Pe}_T> 1,\\
&\frac{\delta_1 h^2_{T}}{\eps}  &&\quad \hbox{ if } \mathsf{Pe}_T\leq 1,
\end{aligned}
\right.\quad\text{and} \quad\delta_T=\min\{\widetilde{\delta_T},c^{-1}\}, \label{e:2.10}
\end{equation}
with $\displaystyle \mathsf{Pe}_T:=\frac{h_T \|\mathbf{w}_h\|_{L^\infty(\Gamma_T)}}{2\eps}$, the usual background tetrahedral mesh size $h_{T}$, and some given positive constants  $\delta_0,\delta_1\geq 0$.

 Define $\delta(x)=\delta_T$ for $x\in \Gamma_T$. The discretization error of the trace SUPG method \eqref{FEM_SUPG} can be estimated in the following mesh-dependent norm:
\begin{equation} \label{defn}
 \| u \|_{\ast}:=\left(\eps\int_{\Gamma_h}|\nath u|^2\, {\rm ds} + \int_{\Gamma_h}\delta(x)|\bw_h\cdot\nath u|^2\, {\rm ds} + \int_{\Gamma_h}c \, |u|^2\, {\rm ds} \right)^{\frac{1}{2}}.
\end{equation}
With the further assumption $\Div_{\Gamma}\bw=0$, the following  error estimate is proved in \cite{ORX}:
 \begin{equation*}
\|u^e-u_h\|_{\ast}\lesssim h\big(h^{1/2}+\eps^{1/2}+c^{\frac12}_{\max}h+\frac{h}{\sqrt{\eps+ c_{\min}}}+ \frac{h^3}{\sqrt{\eps}}\big)  (\|u\|_{H^2(\Gamma)}+\|f\|_{L^2(\Gamma)}),
\end{equation*}
with $c_{\min}:=\text{ess\,inf}_{x\in \Gamma} c(x)$ and $c_{\max}:=\text{ess\,sup}_{x\in \Gamma} c(x)$.

The SUPG stabilization can be combined with any of the algebraic stabilizations described in section~\ref{sectanalysis}.
Note that the ghost penalty stabilization is often sufficient to stabilize a finite element
method for the convection dominated problems~\cite{2015arXiv151102340B} and then the SUPG method is not needed. 
On the other hand, SUPG stabilization does not change the stiffness matrix fill-in and can be used for higher-order trace finite elements.

\section{\textit{A posteriori} error estimates and  adaptivity} \label{sectadap} In finite element methods, \textit{a posteriori} error estimates play a central role in providing a finite element user with reliable local error indicators. Given  elementwise indicators of the discretization error one may consider certain mesh adaptation strategies. This is a well established approach for problems  where the solution behaves differently in different parts of the domain, e.g. the solution has local singularities. Such a technique is also useful for the numerical solution of PDEs defined on surfaces, where the local behaviour of the solution may depend on physical model parameters as well as on the surface geometry.

\textit{A posteriori} error estimates for the TraceFEM have been derived for the Laplace--Beltrami problem in \cite{DemlowOlshanskii12}
and for the convection--diffusion problem on a stationary surface in \cite{chernyshenko2015adaptive}. In both  papers,
only the case of $k=m=1$ was treated (paper \cite{chernyshenko2015adaptive} dealt with trilinear background elements on octree meshes) and only residual type error indicators have been studied.  One important conclusion of these studies is that reliable and efficient residual error indicators can be based on background mesh characteristics. More precisely, for the TraceFEM solution of the Laplace--Beltrami problem \eqref{discr1} one can  define a family of elementwise error indicators
\begin{multline}
\label{intro1}
\eta_p(T)= C_p  \Big ( |\Gamma_T|^{1/2-1/p} h_T^{2/p} \|f_h+ \Delta_{\Gamma_h} u_h\|_{L_2(T)}
\\  + \sum_{ E \subset \partial \Gamma_T} |E|^{1/2-1/p} h_T^{1/p} \|\llbracket \nabla_{\Gamma_h} u_h \rrbracket \|_{L_2(E)} \Big ),\quad p \in [2, \infty],
\end{multline}
for each  $T\in\T_h^\Gamma$. Here $h_T$ is the diameter of the outer tetrahedron  $T$. In \cite{DemlowOlshanskii12}, for $p < \infty$,   reliability up to geometric terms is shown of the corresponding \textit{a posteriori} estimator that is  obtained by suitably summing these local contributions over the mesh. Numerical experiments with surface solutions experiencing point singularities  confirm the reliability and efficiency of the error indicators for any $2 \le p \le \infty$.
Employing a simple refinement strategy based on $\eta_p(T)$  for the TraceFEM was found to provide optimal-order convergence in the $H^1$ and $L^2$ surface norms, and the choice of $p$ in \eqref{intro1} had essentially no effect on the observed error decrease even with respect to constants. This is another example of the important principle that the properties of the TraceFEM are driven by the properties of the background elements.

Below we set $p=2$, i.e., only the properties of the background meshes are taken into account, and formulate a  result
for the case of a convection--diffusion problem.
For each surface element $\Gamma_T$, $T\in\mathcal{T}_h^\Gamma$, denote by ${\mathbf{m}_h}|_{E}$ the outer normal to an edge $E\in\partial \Gamma_T$ in the plane which contains the element $\Gamma_T$. For a curved surface $\Gamma$, `tangential' normal vectors to $E$ from two  different sides are not necessarily collinear.  Let $\llbracket\mathbf{m}_h\rrbracket|_{E}=\mathbf{m}_h^+  +\mathbf{m}_h^-$ be the jump of two outward normals on the edge $E$. For a planar surface, this jump is zero. Over $\Gamma_h$, these jumps produce an additional consistency
 term in the integration by parts  formula and so they end up in the residual error indicators as shown below.

Consider the TraceFEM error $u-u_h^l$ ($u_h^l$ is the TraceFEM solution lifted on $\Gamma$, i.e., $(u_h^l)^e=u_h$),  with $s_h(\cdot,\cdot)=0$, $k=m=1$ and $a_h(\cdot,\cdot)$ as in \eqref{FEM_SUPG} with $\delta_T=0$ (no SUPG stabilization).
The functional $\|[v]\|:=(\eps \|\nabla_\Gamma v\|_{L_2(\Gamma)}^2+\|(c+\frac12\Div_\Gamma\bw)\, v\|_{L_2(\Gamma)}^2)^{\frac12}$ defines a norm of $V_h^\Gamma$. The following \textit{a posteriori} bound  can be proved, cf. \cite{chernyshenko2015adaptive}:
\begin{equation}\label{apost}
\|[u-u_h^l]\|\lesssim
\left(\sum_{T\in \T_h^\Gamma}\left[\eta_R(T)^2+\eta_E(T)^2 \right]\right)^{\frac12} +\text{h.o.t.}.
\end{equation}
with
\[
\begin{split}
\eta_R(T)^2&=h^2_{T}\|f_h+\eps\Delta_{\Gamma_h}  u_h-(c_h+\Div_{\Gamma_h}\bw_h) u_h - \mathbf{w}_h\cdot\nath u_h\|^2_{L^2(\Gamma_T)}.\\
\eta_E(T)^2&=\sum_{E\in\partial \Gamma_T}h_{T}\left(\|\llbracket \eps\unablah u_h \rrbracket\|_{L^2(E)}^2+\|\bw_h\cdot\llbracket\mathbf{m}_h\rrbracket\|_{L^2(E)}^2\right).
\end{split}
\]
The ``h.o.t.'' stands for certain geometric and data approximation terms, which are of  higher order with respect to the bulk mesh discretization parameter if $\Gamma$ is smooth and $\Gamma_h$ resolves $\Gamma$ as discussed in section~\ref{sectStrang}. A representation of ``h.o.t.'' in  terms of geometric quantities is given in \cite{DemlowOlshanskii12,chernyshenko2015adaptive}.
\begin{remark}[Numerical experiments] \label{remnumexp2} \rm Results of numerical experiments demonstrating optimal convergence in $H^1$ and $L^2$ surface norms of the adaptive TraceFEM ($k=m=1$) for the Laplace--Beltrami equation with point singularity are found in  \cite{DemlowOlshanskii12}. More numerical examples  for the Laplace--Beltrami and convection--diffusion problems are given in  \cite{chernyshenko2015adaptive}. All experiments reveal similar adaptive properties of the TraceFEM to those  expected from a standard (volumetric) adaptive  FEM.
\end{remark}

\section{Coupled surface-bulk problems} \label{sectbulksurface}
Coupled bulk-surface or bulk-interface partial differential equations arise in many applications, e.g.,  in multiphase fluid dynamics \cite{GReusken2011} and biological applications \cite{Bonito}.
In this section, we consider a relatively simple coupled bulk-interface advection-diffusion problem. This problem arises in models describing the behavior of soluble surface active agents (surfactants) that are adsorbed at liquid-liquid interfaces.
For a discussion of physical phenomena related to soluble surfactants in two-phase incompressible flows we refer to the literature, e.g., \cite{GReusken2011,Ravera,Clift,Tasoglu}.

Systems of partial differential equations that couple bulk domain effects with interface (or surface) effects pose  challenges both for the  mathematical analysis of equations and the development of numerical methods. These  challenges grow if  phenomena  occur at different physical scales, the coupling is nonlinear or the interface is evolving in time. To our knowledge, the analysis of numerical methods for coupled bulk-surface (convection-)diffusion has been addressed in the literature only very recently. In fact, problems related to the one addressed in this section have been considered only in \cite{burman2016cut,eigel2017posteriori,ElliottRanner,GrossOlshanskiiReusken2015}. 
In these references finite element methods for coupled bulk-surface partial differential equations are proposed and analyzed. In \cite{burman2016cut,ElliottRanner} a stationary diffusion problem on a bulk domain is linearly coupled with a stationary diffusion equation on the boundary of this domain. A key difference between the methods in \cite{burman2016cut} and \cite{ElliottRanner} is that in the latter boundary \emph{fitted} finite elements are used, whereas in the former \emph{unfitted} finite elements  are applied. Both papers include error analyses of these methods. In the recent paper \cite{Chen2014} a similar coupled surface-bulk system is treated with a  different approach, based on the immersed boundary method. In that paper an evolving surface is considered, but only spatially two-dimensional problems are treated and no theoretical error analysis is given. The TraceFEM that we treat in this section is the one presented in \cite{GrossOlshanskiiReusken2015}. We restrict to stationary problems and a
linear coupling between a surface/interface PDE and convection--diffusion  equations in the two adjacent subdomains. The results obtained are a starting point for research on other classes of problems, e.g., with an evolving interface.

In the finite element method that we propose, we use the trace technique presented in section~\ref{sectbasic} for discretization of a convection--diffusion equation on the stationary interface. We also apply the trace technique for the discretization of the PDEs in the two bulk domains. In the literature such trace techniques on bulk domains are usually called \emph{cut} finite element methods, cf., e.g.,~\cite{burman2016cut} and section~\ref{secOther}. As we will see below in section~\ref{sectmethodbulk}, we can use \emph{one} underlying standard finite element space, on a triangulation which is not fitted to the interface, for the discretization of \emph{both} the interface and bulk PDE. This leads to a conceptually very simple approach for treating such coupled problems, in particular for applications with an evolving interface.

The results in the remainder of this section are essentially taken from \cite{GrossOlshanskiiReusken2015}. We restrict to a presentation of the key points and refer to \cite{GrossOlshanskiiReusken2015} for further information.

\subsection{Coupled bulk-interface mass transport model} \label{sectmodel}
We outline the physical background of the coupled bulk-interface model that we treat.
 Consider a two-phase incompressible flow system in which two immiscible  fluids occupy   subdomains $\Omega_i(t)$, $i=1,2$, of a given domain $\Omega \subset \mathbb{R}^3$. The outward pointing normal from $\Omega_1$ into $\Omega_2$ is denoted by $\bn$, $\bw(\bx,t)$, $\bx \in \Omega$, $t\in [0,T]$ is the fluid velocity.  The sharp interface between the two fluids is denoted by $\Gamma(t)$.
The interface is passively  advected with the flow. Consider a surfactant that is soluble in both phases and can be adsorbed and desorbed at the interface.   The surfactant \emph{volume} concentration is denoted by $u$, $u_i=u|_{\Omega_i}$, $i=1,2$. The surfactant \emph{area} concentration on $\Gamma$ is denoted by $v$. Change of the surfactant concentration happens due to convection by the velocity field $\bw$, diffusive fluxes in $\Omega_i$, a diffusive flux on $\Gamma$ and fluxes coming from adsorption and desorption. The net flux (per surface area) due to adsorption/desorption between $\Omega_i$ and $\Gamma$ is denoted by $j_{i,a}-j_{i,d}$.  The conservation of mass in a control volume 
results in the following system of \emph{coupled bulk-interface convection--diffusion equations}, where we use dimensionless variables and $\dot u$ denotes the material derivative of $u$:
\begin{equation*} \label{total}
 \begin{split}
   \dot u_i- \eps_i \Delta u_i  & = 0 \quad \text{in}~ \Omega_i(t), ~i=1,2, \\
\dot v + (\Div_\Gamma \bw) v - \eps_\Gamma \Delta_\Gamma v  & =- K [ \eps \bn \cdot \nabla u]_\Gamma   \quad \text{on}~~\Gamma(t), \\
  (-1)^{i} \eps_i \bn \cdot \nabla u_i & = j_{i,a}-j_{i,d} \quad \text{on}~~\Gamma(t), \quad i=1,2.
\end{split}
\end{equation*}
Here $K$ is a strictly positive constant and $\eps_i$, $\eps_\Gamma$ are the diffusion constants.
A standard  constitutive relation for modeling the adsorption/desorption is as follows:
\begin{equation*} \label{eq4}
 j_{i,a}-j_{i,d}=  k_{i,a} g_i(v) u_i - k_{i,d} f_i(v),\quad\text{on}~\Gamma,  
\end{equation*}
with $k_{i,a}$, $k_{i,d}$ positive coefficients.  Basic choices for $g$, $f$ are the following
$
g(v)=1,~ f(v)=v~\text{(Henry)},\quad g(v)=1- \frac{v}{v_{\infty}},~f(v)=v~\text{(Langmuir)}.
$
Further options are given in \cite{Ravera}. The resulting model is often used in the literature for describing surfactant behavior, e.g.~\cite{Eggleton,Tasoglu,Chen2014}.

 We consider a further simplification of this model and restrict to the Henry constitutive law $ g(v)=1$,  assume $\Gamma$ to be stationary, i.e., $\bw \cdot \bn=0$ on $\Gamma$ and  $\frac{\partial u}{\partial t}=\frac{\partial v}{\partial t} =0$. After a suitable transformation, which reduces the number of parameters, one obtains the following stationary model problem:
 \begin{equation} \label{Totaltrans}
 \begin{aligned}
 - \eps_i\Delta u_i + \bw \cdot \nabla u_i  & = f_i \quad \text{in}~ \Omega_i, ~i=1,2, \\
 - \eps_\Gamma \Delta_\Gamma v + \bw \cdot \unabla v + K[ \eps \bn \cdot \nabla u]_\Gamma& = g  \quad \text{on}~~\Gamma, \\
  (-1)^{i} \eps_i \bn \cdot \nabla u_i & = u_i- q_i v \quad \text{on}~~\Gamma, \quad i=1,2, \\  \bn_{\Omega} \cdot \nabla u_2  & =0 \quad \text{on}~~\partial \Omega, \\ \text{with} ~~q_i & = \frac{ k_{i,d}}{k_{1,a}+ k_{2,a}} \in [0,1].
\end{aligned}
\end{equation}
The data $f_i$ and $g$ must satisfy the consistency condition
\begin{equation} \label{consfg}
  K\big( \int_{\Omega_1} f_1 \, d\bx + \int_{\Omega_2} f_2 \, d\bx\big) + \int_\Gamma g \, ds =0.
 \end{equation}
\noindent\textbf{Well-posed weak formulation}. As a basis for the TraceFEM we briefly discuss a well-posed weak formulation of the model bulk-surface model problem \eqref{Totaltrans}. We introduce some further notation. For $u \in H^1(\Omega_1 \cup \Omega_2)$ we also write $u=(u_1,u_2)$ with $u_i= u_{|\Omega_i} \in H^1(\Omega_i)$. We use the following scalar products:
\begin{align*}
(f,g)_\omega & := \int_{\omega}fg\,dx,~~~\|f\|_\omega^2:=(f,f)_\omega,\quad ~\omega \in \{\Omega,\Omega_i,\Gamma\},\\
(\nabla u,\nabla w)_{\Omega_1 \cup \Omega_2} &:=\sum_{i=1,2}\int_{\Omega_i}\nabla u_i \cdot \nabla w_i\,dx, \quad u,w \in H^1(\Omega_1 \cup \Omega_2).
\end{align*}
In the original (dimensional) variables a natural condition is conservation of total mass, i.e. $(u_1,1)_{\Omega_1}+(u_2,1)_{\Omega_2}+(v,1)_\Gamma=m_0$, with $m_0 >0$ the initial total mass. Due to the transformation of variables we obtain the corresponding natural gauge condition
\begin{equation}\label{Gau}
 K (1+r)(u_1,1)_{\Omega_1}+ K(1+\frac{1}{r})(u_2,1)_{\Omega_2} + (v,1)_\Gamma=0, \quad r= \frac{ k_{2,a}}{k_{1,a}}.
\end{equation}
Define the product spaces
\begin{align*}
  \bV & =H^1(\Omega_1 \cup \Omega_2) \times H^1(\Gamma), \quad \|(u,v)\|_{\bV}=\big( \|u\|_{1, \Omega_1 \cup \Omega_2}^2 +\|v\|_{1,\Gamma}^2\big)^\frac12, \\ \bVt & = \{\, (u,v) \in \bV\,:\,(u,v)~~\text{satisfies}~\eqref{Gau}\,\}.
\end{align*}

To obtain the weak formulation, we multiply the bulk and surface equation in \eqref{Totaltrans} by test functions from $\bV$, integrate by parts and use interface and boundary conditions. The resulting weak formulation
reads: Find $(u,v)\in \bVt$ such that for all $(\eta,\zeta)\in \bV$:
\begin{align}
 a((u,v);(\eta,\zeta)) & = (f_1,\eta_1)_{\Omega_1}+(f_2,\eta_2)_{\Omega_2} +(g,\zeta)_{\Gamma} ,\label{weak} \\
  a((u,v);(\eta,\zeta)) &:= (\eps\nabla u,\nabla \eta)_{\Omega_1 \cup \Omega_2} + (\bw \cdot \nabla u,\eta)_{\Omega_1\cup \Omega_2} + \eps_\Gamma(\unabla v,\unabla \zeta)_{\Gamma}  \nonumber \\  & + (\bw \cdot \unabla v,\zeta)_\Gamma
   +  \sum_{i=1}^2(u_i-q_iv, \eta_i- K \zeta)_\Gamma. \nonumber
\end{align}
In \cite{GrossOlshanskiiReusken2015} the following well-posedness result is proved.
\begin{theorem}\label{Th_wp} For any $f_i\in L^2(\Omega_i)$, $i=1,2$, $g\in L^2(\Gamma)$ such that \eqref{consfg} holds, there exists a unique
solution $(u,v) \in \bVt$ of \eqref{weak}. This solution satisfies the a-priori estimate
\begin{equation*}\label{apr_est1}
\|(u,v)\|_{\bV}\le C \|(f_1,f_2,g)\|_{\bV'}\le c (\|f_1\|_{\Omega_1} +\|f_2\|_{\Omega_2} +\|g\|_{\Gamma}),
\end{equation*}
with  constants $C,c$ independent of $f_i$, $g$ and $q_1,q_2 \in [0,1]$.
\end{theorem}
\subsection{Trace Finite Element Method} \label{sectmethodbulk}
In this section we explain a TraceFEM for the discretization of the problem \eqref{weak}. We assume an interface approximation based on the level set function as in \eqref{eq:Gamma}, \eqref{eq:approx_phi}, \eqref{eq:Gamma_h}, i.e., for the interface approximation we take:
\begin{equation} \label{eq:Gamma_hA}
\Gamma_h= \{\, x \in \Omega\,:\,\phi_h(x)=0\,\}, \quad \text{with}~~\phi_h \in V_{h,k}.
\end{equation}
Note that for $k=1$ (linear FE approximation $\phi_h$ of $\phi$) this $\Gamma_h$ is easy to compute, but for $k>1$ the (approximate) reconstruction of $\Gamma_h$ is a non-trivial problem, cf. the discussion in section~\ref{sectGammaapprox}. Furthermore  we introduce the bulk subdomain approximations
\begin{equation*}
\Omega_{1,h}:= \{\, \bx \in  \Omega\,:\,\phi_h(\bx) < 0\, \},\quad \Omega_{2,h}:= \{\, \bx \in  \Omega\,:\,\phi_h(\bx) > 0\, \}.
\end{equation*}
From \eqref{eq:approx_phi} and properties of $\phi$ it follows that
$
{\rm dist}(\Gamma_h , \Gamma)\leq c h^{k+1}
$
holds, cf.~\eqref{resdist}.
We use the standard space of all continuous piecewise polynomial functions of  degree $m \ge 1$ with respect to a shape regular triangulation  $\T_h$ on $\Omega$, cf.~\eqref{defVm}: $V_{h}^{\rm bulk}:=V_{h,m}$.
We now define three \emph{trace spaces} of finite element functions:
\begin{equation*}\label{FEspace}
\begin{aligned}
V_{h}^\Gamma &:=\{v\in C(\Gamma_h)\,:\, v=w|_{\Gamma_h}~~\text{for some}~w\in V_{h}^{\rm bulk}\},\\
V_{i,h}&:=\{v\in C(\Omega_{i,h})\,:\, v=w|_{\Omega_{i,h}}~~\text{for some}~w\in V_{h}^{\rm bulk}\},~~i=1,2.
 \end{aligned}
\end{equation*}
We need the  spaces $V_{\Omega,h}=V_{1,h}\times V_{2,h}$ and $\bV_{h}= V_{\Omega,h}\times V_{h}^\Gamma \subset H^1(\Omega_{1,h} \cup \Omega_{2,h}) \times H^1(\Gamma_h)$. The space $V_{\Omega,h}$ is studied in many papers on the so-called  cut finite element method or XFEM \cite{Hansbo02,Hansbo04,Belytschko03,Fries}. The trace space $V_{h}^\Gamma$  is the surface trace space treated in section~\ref{sectbasic}.

We consider the finite element bilinear form on $\bV_{h}\times\bV_{h}$, which results from the bilinear form of
the differential problem using integration by parts in advection terms and further replacing $\Omega_{i}$ by $\Omega_{i,h}$ and
$\Gamma$ by $\Gamma_h$:
\[
 \begin{split}
a_h((u,v);(\eta,\zeta)) &= \sum_{i=1}^2\left\{ \eps_i (\nabla u,\nabla \eta)_{\Omega_{i,h}} + \frac12\left[(\bw_h \cdot \nabla u,\eta)_{\Omega_{i,h}}-(\bw_h \cdot \nabla \eta, u)_{\Omega_{i,h}}\right]\right\}\\  & + \eps_\Gamma(\unablah v,\unablah \zeta)_{\Gamma_h} + \frac12\left[(\bw_h \cdot \unablah v,\zeta)_{\Gamma_h}-(\bw_h \cdot \unablah \zeta,v)_{\Gamma_h}\right]\\  & + \sum_{i=1}^2(u_i- q_i v,\eta_i-K\zeta)_{\Gamma_h}.
\end{split}
\]
In this formulation we use the same quantities as in \eqref{weak}, but with $\Omega_i$, $\Gamma$ replaced by $\Omega_{i,h}$ and $\Gamma_h$, respectively.
Let $g_h\in L^2(\Gamma_h)$, $f_h\in L^2(\Omega)$ be given and satisfy
$
K(f_h,1)_{\Omega}+(g_h,1)_{\Gamma_h}=0.
$
As discrete gauge condition we introduce, cf. \eqref{Gau},
\begin{equation*} \label{Gaugediscr}
 K(1+r)(u_h,1)_{\Omega_{1,h}}+ K(1+\frac{1}{r})(u_h,1)_{\Omega_{2,h}} + (v_h,1)_{\Gamma_h}=0, \quad r= \frac{ k_{2,a}}{k_{1,a}}.
\end{equation*}
Furthermore, define
\[
\bV_{h,\alpha}:=\{\, (\eta,\zeta)\in \bV_h \,:\,\alpha_1 (\eta,1)_{\Omega_{1,h}} + \alpha_2 (\eta,1)_{\Omega_{2,h}} + (\zeta,1)_{\Gamma_h}=0\},\quad
\]
for arbitrary (but fixed) $\alpha_1, \alpha_2\ge 0$, and $\bVt_{h}:=\bV_{h,\alpha},$ with $\alpha_1=K(1+r)$, $\alpha_2=K(1+\frac{1}{r})$.
The TraceFEM is as follows: Find $(u_h,v_h)\in \bVt_h$ such that
\begin{equation} \label{weakh}
a_h((u_h,v_h);(\eta,\zeta))=(f_h,\eta)_{\Omega}+(g_h,\zeta)_{\Gamma_h}\quad\text{for all}~(\eta,\zeta)\in\bV_h.
\end{equation}
\noindent\textbf{Discretization error analysis}. In \cite{GrossOlshanskiiReusken2015} an error analysis of the TraceFEM \eqref{weakh} is given. Below we give a main result and discuss the key ingredients of the analysis.
In the finite element space we use the norm given by
\[ \|(\eta,\zeta)\|_{\bV_h}^2:= \|\eta\|_{H^1(\Omega_{1,h} \cup \Omega_{2,h})}^2+ \|\zeta\|_{H^1(\Gamma_h)}^2, \quad (\eta,\zeta) \in H^1(\Omega_{1,h} \cup \Omega_{2,h}) \times H^1(\Gamma_h).
\]
We need smooth extension $u^e$ of $u$ and $v^e$ of $v$. For the latter we take the constant extension along normals as in \eqref{notat} and $u^e$ is taken as follows. We denote by $E_i$ a linear bounded extension operator $H^{k+1}(\Omega_i)\to H^{k+1}(\mathbb{R}^3)$. For a piecewise smooth function $u\in H^{k+1}(\Omega_1\cup\Omega_2)$,  we denote by $u^e$ its
``transformation'' to a  piecewise smooth function $u^e\in H^{k+1}(\Omega_{1,h}\cup\Omega_{2,h})$ defined by
\begin{equation*} \label{defexti}
u^e=\left\{\begin{array}{ll}
E_1(u|_{\Omega_1})&\quad \text{in}~~  \Omega_{1,h}\\
E_2(u|_{\Omega_2})&\quad \text{in}~~  \Omega_{2,h}.
\end{array}\right.
\end{equation*}
The main discretization error estimate is given in the next theorem.
\begin{theorem}\label{Th_Conv1}  Let the solution $(u,v)\in\bVt$ of \eqref{weak} be sufficiently smooth. For the finite element solution $(u_h,v_h)\in\bVt_h$ of \eqref{weakh} the following error estimate holds:
\begin{equation}\label{err_estH1}
\|(u^e-u_h,v^e-v_h)\|_{\bV_h}  \lesssim h^m \big(\|u\|_{H^{m+1}(\Omega)} + \|v\|_{H^{m+1}(\Gamma)}\big)  + h^k \big(\|f\|_{\Omega} + \|g\|_{\Gamma}\big),
\end{equation}
where  $m$ is the degree of the finite element polynomials and $k$ the geometry approximation order  defined in \eqref{eq:Gamma_hA}.
\end{theorem}
\ \\[1ex]
An optimal order $L^2$-norm estimate is also given in \cite{GrossOlshanskiiReusken2015}. We outline the key ingredients used in the proof of Theorem~\ref{Th_Conv1}.

A \emph{continuity estimate} is straightforward: There is a constant $c$ independent of $h$ such that
\begin{equation}\label{conth}
a_h((u,v);(\eta,\zeta))\le c\|(u,v)\|_{\bV_h}\|(\eta,\zeta)\|_{\bV_h }
\end{equation}
for all $(u,v), (\eta,\zeta)\in H^1(\Omega_{1,h}\cup \Omega_{2,h})\times H^1(\Gamma_h)$.
A \emph{ discrete inf-sup stability result}  can be derived along the same lines as for the continuous problem:
For any  $q_1,q_2 \in [0,1]$, there exists  $\alpha$  such that
\begin{equation}\label{infsup2h}
\inf_{(u,v)\in\bVt_h}\sup_{(\eta,\zeta)\in\bV_{h,\alpha}}\frac{a_h((u,v);(\eta,\zeta))}{\|(u,v)\|_{\bV_h}\|(\eta,\zeta)\|_{\bV_h}}
\ge C_{st}>0,
\end{equation}
with a positive constant $C_{st}$ independent of $h$ and of $q_1,q_2 \in [0,1]$.

For the analysis of the consistency error (geometry approximation) we need to be able to compare functions on the subdomains $\Omega_i$ and the interface $\Gamma$ to their  corresponding approximations on $\Omega_{i,h}$ and $\Gamma_h$. For this one needs a ``suitable'' bijection $\Phi_h\,: \Omega \to \Omega$ with the property $\Phi_h(\Omega_{i,h})=\Omega_i$. Such a mapping is constructed in \cite{GrossOlshanskiiReusken2015}. It has the smoothness properties   $\Phi_h\in H^{1,\infty}(\Omega)^3$,
$\Phi_h\in H^{1,\infty}(\Gamma_h)^3$.  Furthermore, for $h$ sufficiently small the estimates
\begin{equation}\label{mapping}
\|\mbox{\rm id}-\Phi_h\|_{L^\infty(\Omega)}+ h \|\bI-\mathrm{D}\Phi_h\|_{L^\infty(\Omega)} + h\|1-\mathrm{det}(\mathrm{D}\Phi_h)\|_{L^\infty(\Omega)} \le c\,h^{k+1}
\end{equation}
hold, where $\mathrm{D}\Phi_h$ is the Jacobian matrix. This mapping is crucial in the analysis of the consistency error. The function $u_i \circ \Phi_h$ defines an extension of $u_i\in H^1(\Omega_i)
$ to $u_i^{\rm ex} \in H^1(\Omega_{i,h})$, which has (even for $u_i \in H^m(\Omega_i)$ with $m >1$) only the (low) smoothness property $H^{1, \infty}(\Omega_{i,h})$. This is not sufficient for getting higher order interpolation estimates. One can, however, show that $ u \circ \Phi_h$ is close to the smooth extension $u^e$, introduced above, in the following sense:
\begin{align}\label{differ_est}
\|u\circ \Phi_h-u^e\|_{\Omega_{i,h}} & \le c h^{k+1} \|u\|_{H^1(\Omega_i)},
 \\
\label{differ_est2}
\|(\nabla u)\circ \Phi_h-\nabla u^e\|_{\Omega_{i,h}} & \le c h^{k+1}  \|u\|_{H^2(\Omega_i)},
 \\
\label{differ_est3}
\|u\circ \Phi_h-u^e\|_{\Gamma_h} & \le c h^{k+1}  \|u\|_{H^2(\Omega_i)},
\end{align}
for $i=1,2,$ and for all $u\in H^2(\Omega_i)$.

Now let $(u,v) \in \bVt$ be the solution of the weak formulation \eqref{weak} and $(u_h,v_h)\in\bVt_h$ the discrete solution of \eqref{weakh}, with suitable data extension (cf.~\cite{GrossOlshanskiiReusken2015}) $f_h$ and $g_h$. We use a compact notation $U:=(u,v)=(u_1,u_2,v)$ for the solution of \eqref{weak}, and similarly $U^e=(u^e,v^e)$, $U_h:=(u_h,v_h)\in \bV_h$ for the solution of \eqref{weakh}, $\Theta=(\eta,\zeta) \in H^1(\Omega_{1,h}\cup \Omega_{2,h}) \times H^1(\Gamma_h)$.
We then get the following \emph{approximate Galerkin relation}:
\begin{equation} \label{ApproxGal}
 a_h(U^e-U_h;\Theta_h)=F_h(\Theta_h):= a_h(U^e;\Theta_h)-a(U;\Theta_h \circ \Phi_h^{-1}) \quad \text{for all}~~\Theta_h \in \bV_h.
\end{equation}
Using the properties \eqref{mapping}, \eqref{differ_est}-\eqref{differ_est3} of the mapping $\Theta_h$ and techniques very similar to the ones used in the consistency error analysis in section~\ref{sectconsist} one can show that, provided the solution $(u,v)$ of \eqref{weak} has smoothness $u\in H^2(\Omega_1\cup\Omega_2)$, $v\in H^2(\Gamma)$, the  following holds:
\begin{equation*}\label{consist_est}
|F_h(\Theta)|\le c h^{k} (\|f\|_\Omega+\|g\|_\Gamma\big) \big(\|\eta \|_{H^{1}(\Omega_{1,h} \cup \Omega_{2,h})}+\|\zeta\|_{H^1(\Gamma_h)} \big)
 \end{equation*}
for all   $\Theta=(\eta,\zeta)\in H^{1}(\Omega_{1,h} \cup \Omega_{2,h}) \times H^1(\Gamma_h)$.

Using this consistency error bound, the continuity result \eqref{conth}, the stability estimate \eqref{infsup2h} and suitable interpolation error bounds,  one can apply a standard Strang Lemma and derive an error bound as in \eqref{err_estH1}.
\begin{remark}[Numerical experiments] \rm
 In \cite{GrossOlshanskiiReusken2015} results of an experiment for the method explained above with $m=k=1$ (linear finite elements and piecewise linear interface approximation) are presented which confirm the optimal first order convergence in an $H^1$-norm and optimal second order convergence in the $L^2$-norm.
\end{remark}
\ \\
\noindent
\\[2ex]
{\bf \Large Part II: Trace-FEM for evolving surfaces} \\[2ex]
Partial differential equations posed on evolving surfaces appear in a number of  applications such as two-phase incompressible flows (surfactant transport on the interface) and flow and transport phenomena in biomembranes.
Recently, several numerical approaches for handling such type of problems have been introduced, cf.~\cite{DziukElliottAN}. In this part we consider a  class of parabolic transport problems on smoothly evolving surfaces and treat the TraceFEM for this class of problems.

\section{Weak formulation of PDEs on evolving surfaces}\label{sec:evol}
We consider a class of scalar transport problems that is studied in many papers on surface PDEs. The setting is as follows. Consider a surface $\Gamma(t)$ passively advected by a {given} smooth velocity field $\bw=\bw(x,t)$, i.e. the normal velocity of $\Gamma(t)$ is given by $\bw \cdot \bn$, with
$\bn$ the unit normal on $\Gamma(t)$. We assume that for all $t \in [0,T] $,  $\Gamma(t)$ is a  hypersurface that is  closed ($\partial \Gamma =\emptyset$), connected, oriented, and contained in a fixed domain $\Omega \subset \Bbb{R}^d$, $d=2,3$. We consider $d=3$, but all results have analogs for the case $d=2$. The surface convection--diffusion equation that we consider is given by:
\begin{equation}
\dot{u} + ({\Div}_\Gamma\bw)u -{ \eps_d}\Delta_{\Gamma} u= f\qquad\text{on}~~\Gamma(t), ~~t\in (0,T],
\label{transport}
\end{equation}
 with a prescribed source term $f= f(x,t)$ and homogeneous initial condition $u(x,0)=0$ for $x \in \Gamma_0:=\Gamma(0)$. Here $\dot{u}$ denotes the advective material derivative. 
 Furthermore, note that if $\bw \cdot \bn=0$ (as assumed in section~\ref{sectconvdiff}) we have $\dot{u}= \frac{\partial u}{\partial t}  +\bw \cdot \nabla_\Gamma u $ on a stationary $\Gamma$, cf. \eqref{e:2.2}. The equation \eqref{transport}, with $f\equiv 0$ and a nonzero initial condition, is a standard model for diffusive transport on a surface, with Fick's law for the diffusive fluxes, cf. e.g. \cite{GReusken2011}.

Different weak formulations of \eqref{transport} are known in the literature. For describing these we first introduce some further notation.  The space--time manifold is denoted by
 \[
 \Gs= \bigcup\limits_{t \in (0,T)} \Gamma(t) \times \{t\},\quad  \Gs\subset \Bbb{R}^{4}.
 \]
 We make the smoothness assumptions $\|\bw\|_{L^\infty(\Gs)} < \infty$,  $\|\DivG \bw\|_{L^\infty(\Gs)} < \infty$. Here $\Div_\Gamma =\Div_{\Gamma(t)}$ denotes the tangential divergence on $\Gamma(t)$, $t \in (0,T)$. The standard $H^1$-Sobolev spaces on $\Gamma(t)$ and $\Gs$ are denoted by $H^1(\Gamma(t))$ and $H^1(\Gs)$. In  \cite{Dziuk07} the following weak formulation is studied: determine $u \in H^1(\Gs)$ such that $u(\cdot,0)=u_0$ and for $t \in (0,T)$, a.e.:
  \begin{equation} \label{Dziukweak}
\int_{\Gamma(t)} \dot u v + u v \DivG \bw + \eps_d \unabla u \cdot \unabla v \, ds = 0 \quad \text{for all}~~v \in H^1(\Gamma(t)).
 \end{equation}
Well-posedness of this weak formulation is proved in \cite{Dziuk07}, assuming $u(x,0) \in H^1(\Gamma(0))$. This formulation is the basis for the evolving surface finite element method, developed in a series of papers by Dziuk-Elliott starting from \cite{Dziuk07}. This method is a Lagrangian method where standard surface finite element spaces defined on an approximation of $\Gamma(0)$ are ``transported'' using a discrete approximation $\bw_h(\cdot,t)$ of $\bw(\cdot,t)$ and then used for a Galerkin discretization of \eqref{Dziukweak}. The \emph{space--time TraceFEM } is a Eulerian method that is based on a different weak formulation, that we now introduce.

Due to the identity
\begin{equation}\label{transform}
 \int_0^T \int_{\Gamma(t)} f(s,t) \, ds \, dt = \int_{\Gs} f(s) (1+ (\bw \cdot \bn)^2)^{-\frac12}\, ds,
\end{equation}
the scalar product $(v,w)_0=\int_0^T \int_{\Gamma(t)} v w \, ds \, dt$
induces a norm   that is equivalent to the standard norm on $L^2(\Gs)$. For our purposes,
it is more convenient to consider the $(\cdot,\cdot)_0$ inner product on  $L^2(\Gs)$.
Let $\nablaG= \nabla_{\Gamma(t)}$ denote the tangential gradient for $\Gamma(t)$ and introduce  the Hilbert space
\begin{equation*}
H=\{\, v \in L^2(\Gs)\,:\, \|\nablaG v\|_{L^2(\Gs)} <\infty \, \}, \quad (u,v)_H=(u,v)_0+ (\nablaG u, \nablaG v)_0. \label{inner}
  \end{equation*}
  We consider the material derivative $\dot{u}$ of $u \in H$ as a distribution on $\Gs$:
\[
  \la\dot u,\phi\ra= - \int_0^T \int_{\Gamma(t)} u \dot \phi + u \phi \DivG \bw\, ds \, dt  \quad \text{for all}~~ \phi \in C_0^1(\Gs).
\]
In \cite{ORXsinum} it is shown that $C_0^1(\Gs)$ is  dense in $H$.
 If $\dot{u}$ can be extended to a bounded linear functional on $H$, then $\dot u \in H'$ { and $\langle \dot u, v \rangle = \dot u(v)$ for $v \in H$.} Define the space
\[
  W= \{ \, u\in H\,:\,\dot u \in H' \,\}, \quad \text{with}~~\|u\|_W^2 := \|u\|_H^2 +\|\dot u\|_{H'}^2.
\]
In \cite{ORXsinum}   properties of $H$ and $W$ are analyzed.
Both spaces are Hilbert spaces and smooth functions are  dense in $H$ and $W$.  Furthermore, functions from $W$ have well-defined traces in $L^2(\Gamma(t))$ for $t\in[0,T]$, a.e..
Define
\[
\Wo:=\{\, v \in W\,:\,v(\cdot, 0)=0 \quad \text{on}~\Gamma_0\,\}.
\]
We introduce the symmetric bilinear form
\[
  a(u,v)= \eps_d (\nablaG u, \nablaG v)_0 + (\DivG \bw\, u,v)_0, \quad u, v \in H,
\]
which is  continuous:
$
a(u,v)\le(\eps_d+\alpha_{\infty}) \|u\|_H\|v\|_H$, with $\alpha_{\infty}:=\|\DivG \bw\|_{L^\infty(\Gs)}.
$
The weak space--time formulation of \eqref{transport} reads: Find $u \in \Wo$ such that
\begin{equation} \label{weakformu}
 \la\dot u,v\ra +a (u,v) = (f,v)_0 \quad \text{for all}~~v \in H.
\end{equation}
Well-posedness of \eqref{weakformu} follows from the following lemma derived in \cite{ORXsinum}.
\begin{lemma}\label{la:infsup}
The following properties of the bilinear form $\la\dot u,v\ra + a(u,v)$ hold.\vspace{0.5ex}
\begin{description}
\item{a)} Continuity:
$ | \la\dot u,v\ra + a(u,v)| \le (1+\eps_d+\alpha_{\infty})\|u\|_W \|v\|_H $ for all $u \in W,~v \in H$.
\item{b)} Inf-sup stability:\vspace{-1ex}
 \begin{equation}\label{infsup}
  \inf_{0\neq u \in \overset{\circ}{W}}~\sup_{ 0\neq v \in \overset{\phantom{.}}{H}} \frac{\la\dot u,v\ra + a(u,v)}{\|u\|_W\|v\|_H} \geq c_s>0.
 \end{equation}
\item{c)} The kernel of the adjoint mapping is trivial: 
If $\la\dot u,v\ra + a(u,v)=0$ holds for  all $u \in \Wo$, then $v=0$.
\end{description}
\end{lemma}
\smallskip
As a  consequence of  Lemma~\ref{la:infsup} one obtains:
\begin{theorem} \label{mainthm1}
For any $f\in L^2(\Gs)$, the problem \eqref{weakformu} has a unique solution $u\in \Wo$. This solution satisfies the a-priori estimate
\begin{equation*} \label{stabestimate}
\|u\|_W \le c_s^{-1} \|f\|_0.
\end{equation*}
\end{theorem}

Note that the weak formulation \eqref{weakformu} is on  the \emph{whole space--time manifold} $\Gs$. As a starting point for a finite element Galerkin discretization this may seem not very attractive, because we have a globally coupled space--time  problem. However, we shall see how the space--time TraceFEM leads to a time-stepping algorithm, where only a 2D surface problem is solved on each time step.

\section{Space-time TraceFEM} \label{sectSTTraceFEM}
We take a  partitioning of the time interval:  $0=t_0 <t_1 < \ldots < t_N=T$, with a uniform time step $\Delta t = T/N$. The assumption of a uniform time step is made to simplify the presentation, but is not essential. A time interval is denoted by $I_n:=(t_{n-1},t_n]$. Consider  the partitioning of the space--time volume domain $Q= \Omega \times (0,T] \subset \Bbb{R}^{3+1}$ into time slabs  $Q_n:= \Omega \times I_n$.
The variational formulation in \eqref{weakformu} forms the basis for the space--time TraceFEM that we present in this section. The basic idea is the same as for the TraceFEM explained in
section~\ref{sectTraceFEM}: Per time slab, we use a standard space--time finite element space on a fixed outer triangulation  (which is the tensor product of $\T_h \times I_n$) and then take the trace on an approximation of the space--time surface $\Gs^n$. This trace space is used in a standard DG in time -- CG in space approach applied to \eqref{weakformu}. We first present the method without using any surface approximation, and then address questions related to replacing the exact space--time surface by an approximate one.

\noindent\textbf{Basic form of space--time TraceFEM.}
 Corresponding to each time interval $I_n:=(t_{n-1},t_n]$ we assume a given shape regular tetrahedral triangulation $\T_n$ of the spatial domain $\Omega$. The corresponding spatial mesh size parameter is denoted by $h$. 
Then $\mathcal{Q}_h=\bigcup\limits_{n=1,\dots,N}\T_n\times I_n$ is a subdivision of $Q$ into space--time
prismatic nonintersecting elements. We  call $\mathcal{Q}_h$ a space--time triangulation of $Q$.
Note that this triangulation is not necessarily fitted to the surface $\Gs$. We allow $\T_n$ to vary with $n$ (in practice, during time integration one may wish to adapt the space triangulation) and so
the elements of $\mathcal{Q}_h$ may not  match at $t=t_n$.

For any $n\in\{1,\dots,N\}$, let $V_{n,j}$ be the finite element space of continuous piecewise polynomials of degree $j$  on $\T_n$, cf.~\eqref{defVm}.
We define the \emph{bulk space--time finite element space}:
\begin{equation*} \label{defVn}
W_{\ell,m}:= \{ \,  v(x,t)= \sum_{k=0}^\ell t^k \phi_k(x)~\text{on every}~Q_n,~\text{with}~\phi_k \in V_{n,j},~0\leq k \leq \ell \,\}.
\end{equation*}
This is a standard space--time finite element space on $\mathcal{Q}_h$, with piecewise polynomials that are continuous in space and discontinuous in time. For the well-posedness result and error analysis we define the \emph{surface finite element space} as the space
of traces of functions from $W_h^{\rm bulk}=W_{\ell,m}$ on $\Gs$:
\begin{equation*} \label{deftraceFE}
  W_h^\Gs := \{ \, w:\Gs \to \Bbb{R}\,:\, w=v|_{\Gs}, ~~v \in W_h^{\rm bulk} \, \}.
\end{equation*}
In addition to $a(\cdot,\cdot)$,  we define on  $ W_{\ell,m}^\Gs$ the following bilinear forms:
\begin{align*}
 d(u,v)  =  \sum_{n=1}^N d^n(u,v), \quad d^n(u,v)=([u]^{n-1},v_+^{n-1})_{t_{n-1}},\quad
 \la \dot u ,v\ra_b =\sum_{n=1}^N  \la \dot u_n, v_n\ra,
\end{align*}
and
\[
\la \dot u_h ,v_h\ra_b=\sum_{n=1}^N\int_{t_{n-1}}^{t_n}\int_{\Gamma(t)} (\frac{\partial u_h}{\partial t} +\bw\cdot\nabla u_h)v_h ds\,dt.
\]
The basic form of the space--time TraceFEM is a discontinuous Galerkin method: Find $u_h \in W_h^{\rm bulk}:=W_{\ell,m}$ such that
\begin{equation} \label{brokenweakformu_h}
  \la \dot u_h ,v_h\ra_b +a(u_h,v_h)+d(u_h,v_h) = (f,v_h)_0 \quad \text{for all}~~v_h \in W_h^{\rm bulk}.
\end{equation}
As usual in time-DG methods,  the initial condition for $u_h(\cdot,0)$ is treated in a weak sense.
Obviously the method can be implemented with a time marching strategy. For the implementation of the method one needs an approach to approximate the integrals over $\Gs^n$. This question is briefly  addressed below. Before we come to that, we first introduce a variant of the method in  \eqref{brokenweakformu_h}.

\noindent\textbf{Variant of space--time TraceFEM with stabilization.}
 We first explain a discrete mass conservation property of the scheme \eqref{brokenweakformu_h}. We consider the case that \eqref{transport}  is derived from mass conservation of a scalar quantity with a diffusive flux on $\Gamma(t)$. The original problem then has a nonzero initial condition $u_0$ and a source term $f \equiv 0$.  The solution $u$ of the original problem has the mass conservation property $\bar u(t):=\int_{\Gamma(t)} u \, ds =\int_{\Gamma(0)} u_0 \, ds $ for all $t \in [0,T]$. After a suitable transformation one obtains the equation   \eqref{transport} with a zero initial condition $u_0=0$ and a right hand-side $f$ which has the zero average property $\int_{\Gamma(t)} f \, ds =0 $ for all $t \in [0,T]$. The solution $u$ of \eqref{transport} then has the ``shifted'' mass conservation property $\bar u(t)=0$ for all $t \in [0,T]$.  Taking suitable test functions in the
discrete problem \eqref{brokenweakformu_h} we obtain that the discrete solution $u_h$ has the following (weak) mass conservation property, with $\bar u_h(t):=\int_{\Gamma(t)} u_h \, ds$:
\begin{equation}\label{means}
\bar u_{h,-}(t_{n})=0\quad\text{and}\quad\int_{t_{n-1}}^{t_n} \bar u_h(t)\, dt=0,\quad n=1,2,\dots N.
\end{equation}
Although   \eqref{means} holds,  $\bar u_h(t) \neq 0$ may occur for $ t_{n-1} \le t < t_n$.  We introduce   a \emph{consistent} stabilizing term involving the quantity $\bar u_h (t)$.  More precisely, define
\begin{equation} \label{stabblf}
  a_\sigma(u,v):= a(u,v)+\sigma \int_0^T \bar u(t) \bar v(t) \, dt, \quad \sigma \geq 0.
\end{equation}
Instead of \eqref{brokenweakformu_h} we consider the stabilized version: Find $u_h \in W_h^{\rm bulk}$ such that
\begin{equation} \label{brokenweakformu_h1}
  \la \dot u_h ,v_h\ra_b +a_\sigma(u_h,v_h)+d(u_h,v_h) =(f,v_h)_0 \quad \text{for all}~~v_h \in W_h^{\rm bulk}.
\end{equation}
Taking $\sigma >0$ we expect both a stabilizing effect and an improved discrete mass conservation property, since $\sigma\to\infty$ enforces $\bar u(t)=0$ condition for $t\in[0,T]$. We will explain in section~\ref{sectdiscranalysis} why  the stabilizing term is important for deriving ellipticity of the bilinear form, which is a key ingredient in the error analysis.

\noindent\textbf{Approximation of the space--time surface and matrix-vector representation.}
 Two main implementation issues are the approximation, per time slab, of the space--time integrals in the bilinear form $\la \dot u_h ,v_h\ra_b +a_\sigma(u_h,v_h)$ and the representation of the  finite element trace functions in $ W_h^\Gs$.  One possibility to approximate the integrals, is to make use of the formula \eqref{transform}, converting space--time integrals to surface integrals over $\Gs^n$, and then to
approximate $\Gs^n$ by a ``discrete'' surface $\Gs_h^n$. This is done locally, i.e. time slab per time slab.    In the context of level set methods, we typically have an (accurate) approximation $\phi_h(x,t) \in W_{\ell,k}$, of the level set function $\phi(x,t)$, $t \in I_n$. For  the surface approximation $\Gs_h^n$ one can then take the zero level of $\phi_h$, i.e., we use the space--time analog of \eqref{eq:Gamma_h}:
\begin{equation} \label{eq:Gamma_hST}
\Gs_h^n= \{\, (x,t) \in \Omega\times I_n\,:\,\phi_h(x,t)=0\,\}.
\end{equation}
It is not clear how to represent this surface approximation in a computationally efficient way in the higher order  case $\ell \geq 2$ or $k \geq 2$. This approximation, however is easy to compute for $\ell=k=1$. Then
 $\phi_h$ is a bilinear (in $x$ and $t$) finite element approximation of the level set function $\phi(x,t)$. Within each space--time prism the zero level of  $\phi_h\in W_{1,1}$ can be represented as a union of tetrahedra, cf. \cite{refJoerg}, and standard quadrature formulas can be used.
 Results of numerical experiments with  this treatment of integrals over $\Gs_h^n$ are reported
in \cite{refJoerg,GOReccomas,ORXsinum}. The use of numerical quadratures in time and space separately was suggested in \cite{hansbo2016cut_td} together with  adding a stabilization term to ensure that the resulting problems are well-conditioned. To reduce the ``geometric error'', it may be efficient  to use $\phi_h \in W_{1,1}$ on a finer space--time mesh than the one used in the approximation $u_h$ of $u$, e.g., one additional refinement of the given
outer space--time mesh. Clearly, using this surface approximation, the discretization error can be at most second order. So far, only this type of bilinear space--time surface approximation has been implemented and tested. A higher order surface approximation  method, for example an extension of the isoparametric TraceFEM treated in section~\ref{sectGammaapprox} to a space--time setting, has not been developed, yet.

For the representation of  the finite element  functions in $W_h^\Gs$ it is natural to use traces of the standard nodal basis functions in the volume space--time finite element space $W_{\ell,m}$.   As in the case of TraceFEM on stationary surfaces,  these trace functions in general form (only) a frame in $W_h^\Gs$. A finite element surface solution is represented as a linear combination of the elements from this frame.
Linear systems resulting in every time step  may have more than one solution, but every solution yields the same trace function, which is the unique solution of \eqref{brokenweakformu_h1} in $W_h^\Gs$.  If $\ell=k=1$, $\Delta t \sim h$ and $\|\bw \|_{L^\infty(\Gs^n)} = \mathcal{O}(1)$, then the number of tetrahedra $T \in \T_n$ that are intersected by $\Gamma(t)$, $t \in I_n$, is of the order $\mathcal{O}(h^{-2})$. Hence, per time step the linear systems have $\mathcal{O}(h^{-2})$ unknowns, which is the same complexity as a discretized spatially \emph{two}-dimensional problem. Note that although we derived the method in $\Bbb{R}^{3+1}$, due to the time stepping and the trace operation, the resulting algebraic problems have two-dimensional complexity. Since the algebraic problems have a complexity of (only) $\mathcal{O}(h^{-2})$ it may be efficient to use  a  sparse direct solver for computing the discrete solution. Stabilization procedures, as presented in section~\ref{sectStiffness} for a stationary surface,
  and
further linear algebra aspects of the  space--time TraceFEM have not been studied so far.

The stabilization term   in \eqref{stabblf} does not cause significant additional computational work, as explained in \cite{olshanskii2014error}.

\section{Stability and error analysis  of space--time TraceFEM} \label{sectdiscranalysis}
We outline a framework for the error analysis of the space--time TraceFEM, further details are found in \cite{olshanskii2014error}.
Stability of the method and error bounds are derived with respect to the energy norm:
\[
\enorm{u}_h:=\left(\max_{n=1,\dots,N}\|u_{-}^n\|_{t^n}^2 + \sum_{n=1}^N \|[u]^{n-1}\|_{t_{n-1}}^2+\|u\|_H^2\right)^\frac12.
\]
Using test functions in \eqref{brokenweakformu_h1} that are restrictions of $u_h\in W_h^\Gs$ for time interval $[0,t_n]$ and zero for $t>t_n$ one can derive the following stability result for the bilinear form of the space--time TraceFEM.
\begin{theorem} \label{stabadd}
Assume  $ \sigma \geq  \frac{\eps_d}2\max\limits_{t\in[0,T]}\frac{c_F(t)}{|\Gamma(t)|}$, where $c_F(t)$ is the Poincare--Friedrichs constant for $\Gamma(t)$. Then the following inf-sup estimate holds:
\begin{equation} \label{infsuph}
 \inf_{u_h\in W_h^\Gs} \sup_{v_h\in W_h^\Gs}\frac{\la \dot u_h,v_h\ra_b +a_\sigma(u_h,v_h)+d(u,v)}{\enorm{v_h}_h \enorm{u_h}_h} \geq c_s>0.
\end{equation}
\end{theorem}

The well-posedness of \eqref{brokenweakformu_h1} in the space of traces  and the stability estimate $\enorm{u_h}_h \le c_s^{-1}  \|f\|_{0}$ for the solution $u_h\in W_h^\Gs$ readily follow from Theorem~\ref{stabadd}.
\smallskip

The following observation, which is standard in the theory of discontinuous Galerkin methods, simplifies the discretization error analysis.
Denote by $\Gs^n$ one time slab of the space--time manifold, $\Gs^n:=\cup_{t \in I_n}\Gamma(t)\times\{t\}$,  introduce the following subspaces of $H$:
\[H_n:=\{\, v \in H\,:\,v=0  \quad \text{on}~~\Gs \setminus \Gs^n\, \},
\]
and define the spaces
\begin{align}
  W_n & = \{ \, v\in H_n\,:\,\dot v \in H_n' \,\}, \quad \|v\|_{W_n}^2 = \|v\|_{H}^2 +\|\dot v\|_{H_n'}^2, \notag \\
W^b  &  := \oplus_{n=1}^N W_n,~~\text{with norm}~~ \|v\|_{W^b}^2= \sum_{n=1}^N \|v\|_{W_n}^2. \label{brokenW}
\end{align}
One can show  that the  bilinear form on the left hand side of \eqref{brokenweakformu_h1} is well defined on $W^b\times W^b$. Moreover, the \emph{unique solution of \eqref{weakformu} is also the unique solution} of the  following variational problem in the \emph{broken} space $W^b$: Find $u \in W^b$ such that
\begin{equation} \label{brokenweakformu}
  \la \dot u ,v\ra_b +a(u,v)+d(u,v) =( f,v )_0 \quad \text{for all}~~v \in W^b.
\end{equation}
For this time-discontinuous weak formulation an inf-sup stability result as in \eqref{infsuph} with $W_h^\Gs$ replaced by $W^b$
can be derived. A simplification of the error analysis comes from the observation  that our space--time TraceFEM (without geometry approximation) can be treated as a \textit{conforming} Galerkin FEM for the variational problem \eqref{brokenweakformu}.

As usual, for the error analysis one needs continuity of the TraceFEM bilinear form and of the adjoint bilinear form.
By standard arguments one shows the following results:
\begin{align} \label{cont}
  |\la \dot e,v\ra_b +a_\sigma(e,v)+d(e,v)|&
\leq c  \enorm{v}_h(\|e\|_{W^b}+\sum_{n=0}^{N-1}\|[e]^n\|_{t^n}),\\
|\la \dot e,v\ra_b +a_\sigma(e,v)+d(e,v)|&
\leq c  \enorm{e}_h(\|v\|_{W^b}+\sum_{n=1}^{N-1}\|[v]^n\|_{t^n}+\|v\|_T), \label{cont_dual}
\end{align}
for any $e,v\in W^b$, with constants $c$ independent of $e, v, h, N$.

\noindent\textbf{Extension of functions defined on  $\Gs$.}
Similar to the case of stationary manifolds,  approximation properties of the trace space $W_h^\Gs$ completely rely on  approximation properties of the outer space $W_{\ell,m}$. To exploit the latter, we need a suitable extension procedure for smooth functions on the space--time manifold $\Gs$ to a neighborhood of $\Gs$.
For a function $u \in H^2(\Gs)$ we need an extension $u^e \in H^2(\mathcal{O}_h(\Gs))$, where $\mathcal{O}_h(\Gs)$ is an $h$-neighborhood in $\Bbb{R}^4$ that contains the space--time manifold $\Gs$. A suitable extension $u^e$ can be constructed by extending $u$ along the \emph{spatial} normal direction to $\Gamma(t)$ for every $t\in [0,T]$.  We assume $\Gs$ to be a three-dimensional $C^3$-manifold in $\Bbb{R}^4$. The following result is proved in~\cite{olshanskii2014error}:
\begin{equation}\label{H2equiv}
\|u^e\|_{H^m(\mathcal{O}_\delta(\Gs))}^2  \leq c \delta\|u\|_{H^m(\Gs)}^2 \quad \text{for all}~~u \in H^m(\Gs),~m=0,1,2.
\end{equation}
with
\begin{equation*} \label{defU}
 \mathcal{O}_\delta(\Gs) = \{\, \bx:=(x,t) \in  \R^{3+1}\,:\,{\rm dist}(x,\Gamma(t)) < \delta\, \}.
\end{equation*}

\noindent\textbf{Interpolation and error bounds.}
Recall that the local space--time triangulation $\mathcal{Q}_h^\Gs$ consists of cylindrical elements that are intersected by $\Gs$. The domain formed by these prisms is denoted by $Q^\Gs$.
For $K \in \mathcal{Q}_h^\Gs$, the nonempty intersections are denoted by $\Gs_K=K \cap \Gs$.
Let
\[
  I_h: C(Q^{\Gs}) \to W_{\ell,m}|_{Q^{\Gs}}
\]
be the nodal interpolation operator.
Since the triangulation may vary from  time-slab to time-slab, the interpolant  is in general discontinuous between the time-slabs.

The key ingredients for proving interpolation bounds are the result in \eqref{H2equiv} with $\delta\sim h$, which allows to control volumetric norms by the  corresponding surface norms, and an elementwise trace inequality, which is the 4D analog of \eqref{inHansbo}. Assuming $\Delta t\sim h$, this trace inequality is as follows:
\begin{equation} \label{resl}
  \|v\|_{L^2(\Gs_K)}^2 \leq c  (h^{-1}\|v\|_{L^2(K)}^2+ h\|v\|_{H^1(K)}^2) \quad \text{for all}~~v \in H^1(K),~  K \in \Q_h^{\Gs},
\end{equation}
with a constant $c$, depending only on the shape regularity of the tetrahedral triangulations $\T_n$ and the smoothness of $\Gs$.
The trace inequality \eqref{resl} is proved in \cite{olshanskii2014error}  with one further technical assumption, which is always satisfied if the mesh sufficiently resolves $\Gs$.

Applying the `extend--interpolate--pull back' argument as in section~\ref{sectapprox} one proves the following approximation bounds for $\ell=k=1$ and $\Delta t\sim h$ and sufficiently smooth $u$ defined on  $\Gs$:
\begin{equation} \label{reskk1}
\begin{aligned}
 \sum_{n=1}^N\|u- I_hu^e\|_{H^k(\Gs^n)}^2 &\leq c h^{2(2-k)} \|u\|_{H^2(\Gs)}^2,  \quad k=0,1,\\
 \|u- (I_hu^e)_{-}\|_{t^n} &\leq c h^{2} \|u\|_{H^2(\Gamma(t^n))},~~ n=1,\dots,N,\\
 \|u- (I_hu^e)_{+}\|_{t^n}&\leq c h^{2} \|u\|_{H^2(\Gamma(t^n))},~~ n=0,\dots,N-1.
\end{aligned}
\end{equation}
The constants $c$ are independent of $u, h, N$. To extend the approximation bounds in \eqref{reskk1} to higher order
space--time finite elements ($\ell>1$, $m>1$) we need an estimate as in \eqref{H2equiv} for higher order Sobolev norms. We expect such estimates to be true, but did not work out the details, yet.

Now the inf-sup inequality \eqref{infsuph}, the Galerkin orthogonality for the TraceFEM \eqref{brokenweakformu_h} (recall that it is a conforming method for the auxiliary broken formulation \eqref{brokenweakformu}), combined with the continuity  and approximation results in \eqref{cont} and \eqref{reskk1} imply the following convergence result.
\begin{theorem} \label{thmmain1}
Let $u$ be the solution of \eqref{weakformu} and assume $u \in  H^2(\Gs)$, $u\in H^2(\Gamma(t))$
for all $t\in[0,T]$.  Let $u_h \in W_h$ be the solution of the discrete problem \eqref{brokenweakformu_h1} with a stabilization parameter $\sigma$ as in Theorem~\ref{stabadd}.  The following error bound holds:
\[
 \enorm{u-u_h}_h \leq c h (\|u\|_{H^2(\Gs)}+\sup_{t\in[0,T]}\|u\|_{H^2(\Gamma(t))}),\quad \Delta t\sim h.
\]
\end{theorem}

The error estimate in Theorem~\ref{thmmain1} assumes that all integrals in \eqref{brokenweakformu_h} over the space--time manifold are computed exactly. This assumption has been made in \cite{olshanskii2014error} to simplify the analysis, but it obviously is not a realistic assumption. In practice an approximation of $\Gs$ is used, as discussed in section~\ref{sectSTTraceFEM}. Taking this surface approximation into account in the analysis, would naturally involve estimates of a consistency term as in Strang's lemma in section~\ref{sectStrang}.
We expect that with similar tools as used for the case of a stationary surface, suitable estimates can be derived. Such results, however, are not available, yet.

Denote also by $\|\cdot\|_{-1}$ the norm dual to the $H^1_0(\Gs)$ norm with respect to the $L^2$-duality.
The  next theorem  gives an $O(h^2)$-convergence estimate for the linear space--time TraceFEM.
\begin{theorem} \label{dualthm}  Assume that $\Gs$ is sufficiently smooth and that the assumptions of Theorem~\ref{thmmain1} are satisfied.
Then the following  error estimate holds:
\[
 \|u-u_h\|_{-1} \leq c h^2 (\|u\|_{H^2(\Gs)}+\sup_{t\in[0,T]}\|u\|_{H^2(\Gamma(t))}).
\]
\end{theorem}
The proof uses the Aubin-Nitsche duality argument and invokes the Galerkin orthogonality, the  continuity result in~\eqref{cont_dual} and the error estimate from Theorem~\ref{thmmain1}.  As is usual in the Aubin-Nitsche duality argument, one needs a regularity result for the  problem dual to \eqref{weakformu}. The required regularity result is proved in \cite{olshanskii2014error}.

Note that $O(h^2)$ convergence was derived in a norm weaker than the commonly considered
$L^2$ norm. The reason  is that the proof uses isotropic  polynomial interpolation error
bounds on 4D space--time elements, see \eqref{reskk1}. Naturally, such bounds call for isotropic space--time $H^2$-regularity bounds for the  solution. For parabolic problems, however, such regularity is more restrictive than in an elliptic case, since the solution is generally less regular in time than in space. We were able to overcome this by measuring the error in the weaker $\|\cdot\|_{-1}$-norm.
\begin{remark}[Numerical experiments] \rm Results of numerical experiments with the linear space--time TraceFEM (i.e., bilinear space--time finite elements and bilinear interpolation of the level set function for the space--time surface approximation) are given in \cite{refJoerg,GOReccomas,ORXsinum}. These results confirm the optimal first order $H^1$ error bound given in Theorem~\ref{thmmain1} and also show optimal second order convergence in the  $L^\infty(L^2(\Gamma(t)))$ norm. The theory on well-posedness of the continuous problem and on the discretization error analysis is applicable only to problems with a smooth space--time surface, i.e, topological changes in $\Gamma(t)$ are not allowed. Surfaces with topological changes can be handled very easily with a level set technique and also the space--time TraceFEM can be directly applied to such a problem. In \cite{GOReccomas} results of the space--time TraceFEM applied to a problem with a topological change (``collision of two spheres'') are presented. These
results illustrate that this method is very robust and yields stable results even for large mesh size ($h$ and $\Delta t$) and in cases with topological singularities.
\end{remark}
%
%
\section{Variants of TraceFEM on evolving surfaces} \label{sectVariants}
Several other possibilities to extend the TraceFEM  to evolving surfaces are known in the literature.
A combination of the TraceFEM and the narrow-band FEM was suggested in \cite{deckelnick2014unfitted}, a  characteristic-Galerkin
TraceFEM was studied in \cite{hansbo2015characteristic} and a hybrid, FD in time -- TraceFEM in space, variant was recently
proposed in \cite{olshanskii2016trace}. Here we review these methods and available analysis. Throughout this section,  $u^{n}$ denotes an approximation to $u(t_{n})$ for the time nodes $0=t_0<\dots<t_N=T$.

The trace--narrow-band FEM by  Deckelnick et al. \cite{deckelnick2014unfitted} is based on the level set description of the surface
evolution. In this method, one assumes an approximation to $\Gamma(t)$ at each time node $t_n$ given by
$
\Gamma^n_h=\{\bx\in\mathbb{R}^3\,:\,\phi_h(t_n,x)=0\}
$
and defines the $h$-width narrow strip around $\Gamma^n_h$,
\[
\mathcal{O}_h(\Gamma^n_h)=\{\bx\in\mathbb{R}^3\,:\,|\phi_h(t_n,x)|<h\}.
\]
The finite element level set function $\phi_h$ is assumed sufficiently regular and has to satisfy $|\nabla\phi_h(t_n,x)|\ge c>0$  in a neighborhood of $\Gamma^n_h$.
The trace--narrow-band FEM benefits from the observation that for a test function $\eta$ constant along material pathes, i.e. $\dot{\eta}=0$,
the transport--diffusion equation \eqref{transport} yields the integral identity
\begin{equation}
\frac{d}{dt}\int_{\Gamma(t)}{u}\eta\,ds  +{ \eps_d}\int_{\Gamma(t)}\nabla_{\Gamma} u \cdot \nabla_{\Gamma} \eta\,ds= \int_{\Gamma(t)}f\eta\,ds\qquad t\in (0,T].
\label{deriv}
\end{equation}
One can extend any given time independent $\psi\,: \,\mathcal{O}_h(\Gamma^n_h)\to\mathbb{R}$  along characteristics backward in time
in such a way that the extended function $\eta$ satisfies $\dot{\eta}=0$, $\eta|_{t=t_n}=\psi$.
This motivates the approximation of the time derivative of the surface integral on the left-hand side of \eqref{deriv} by the difference
\[
\frac{d}{dt}\int_{\Gamma(t)}{u}\eta\,ds\approx\frac{1}{\Delta t}\left(\int_{\Gamma(t_{n})}{u}\psi\,ds-\int_{\Gamma(t_{n-1})}{u}\psi(\cdot+\bw^e\Delta t)\,ds\right),\quad \Delta t =t_{n}-t_{n-1}.
\]
To make use of this approximation in the finite element setting one has to define a FE  test function in a neighborhood of $\Gamma(t_n)$.
Trace finite element background functions do not suffice, since $x+\bw^e\Delta t$ may lie out of the strip of tetrahedra intersected by $\Gamma(t_{n})$. This forces one to consider  background FE functions which have nonempty intersection of their support with the narrow  strip $\mathcal{O}_h(\Gamma^{n}_h)$ rather than with $\Gamma^{n}_h$. This leads to the following  finite element formulation: Find $u_h^{n}\in V_h^{\rm bulk}$ satisfying
\begin{multline*}
\frac{1}{\Delta t}\left(\int_{\Gamma^{n}_h}{u}_h^{n}\psi_h\,ds-\int_{\Gamma^{n-1}_h}{u}_h^{n-1}\psi_h(\cdot+\bw^e\Delta t )\,ds\right)\\ +{ \eps_d}\int_{\mathcal{O}_h(\Gamma^{n}_h)}\nabla u^{n}_h \cdot \nabla \psi_h|\, \mbox{det}(\nabla\phi_h(t_n,x))|\,ds= \int_{\Gamma^{n}_h}f_h^n\psi_h\,ds
\end{multline*}
for all $\psi_h\in V_h^{\rm bulk}$. It can be shown~\cite{deckelnick2014unfitted}, that the diffusion term is $O(h^2)$-consistent.
One can also show that the  method is conservative so that it preserves mass in the case of an advec\-tion-diffusion conservation law. The condition $x+\bw^e\Delta t\in \mathcal{O}_h(\Gamma^{n}_h)$, $x\in \Gamma^{n}_h$ implies a Courant type restriction on $\Delta t$. Numerical experiments indicate an $O(\Delta t +h^2)$ accuracy of the method for cases with a smoothly deforming surface, but no rigorous error analysis of the method is available so far.
\smallskip

In an Eulerian description of surface evolution, one typically has no explicit access to trajectories of material points on the surface. However, one may try to reconstruct these numerically based on   the velocity field $\bw$ or its approximation in $\Omega$.
In particular, to approximate $\dot{u}(x)$ at $x\in\Gamma_h^n$ one can use a semi-Lagrangian method to integrate numerically
back in time along the characteristic passing through $x$. Doing this for a time interval $[t_{n-1},t_{n}]$ one finds a point $y$ in a neighborhood of $\Gamma_h^{n-1}$. Due to discretization errors $y\notin\Gamma_h^{n-1}$, in general. Hence, one uses
the closest point projection on $\Gamma_h^{n-1}$ to define the relevant data at $y$. This approach to approximate the material
derivative in \eqref{transport} combined with a $P_1$ TraceFEM to handle the diffusion terms has been studied in     \cite{hansbo2015characteristic}. It is proved that for  $\Delta t \sim h$ this method has first order convergence in the $L^2$ norm. Due to the well-known stability properties of semi-Lagrangian methods, the characteristic-TraceFEM does not need additional stabilization
for problems with dominating transport.
\smallskip

Yet another variant of the TraceFEM for evolving surfaces was recently proposed in \cite{olshanskii2016trace}.
The main  motivation for the method presented in that paper was  to avoid space--time elements or any reconstruction of the space--time manifold. To outline the main idea, assume  that the surface is defined implicitly as the zero level of a smooth
level set function $\phi$ on $\Omega\times(0,T)$:
$
\Gamma(t)=\{\bx\in\mathbb{R}^3\,:\,\phi(t,\bx)=0\},
$
such that $|\nabla\phi|\ge c_0>0$ in  a suitable neighborhood of $\Gs$. One can consider  $u^e$ such that $u^e=u$ on $\Gs$ and $\nabla u^e\cdot\nabla \phi =0$ in the neighborhood of $\Gs$.
Note that $u^e$ is smooth once $\phi$  and $u$ are both smooth. With same notation
$u$  for the solution of the surface PDE  \eqref{transport} and its extension, one obtains the following equivalent formulation of \eqref{transport},
\begin{equation}
\left\{\begin{split}
 \frac{\partial u}{\partial t} + \bw \cdot \nabla u + ({\Div}_\Gamma\bw)u -{ \eps_d}\Delta_{\Gamma} u&=f\qquad\text{on}~~\Gamma(t), \\
 \nabla u\cdot\nabla \phi& =0 \qquad\text{in}~\mathcal{O}(\Gamma(t)).
 \end{split}
 \right.~~t\in (0,T].
\label{transport_new}
\end{equation}
Here $\mathcal{O}(\Gamma(t))$ is a $\mathbb{R}^3$ neighborhood of $\Gamma(t)$ for any fixed $t\in(0,T]$.
Assuming $\Gamma(t_n)$ lies in the neighborhood of $\Gamma(t_{n-1})$, where $u^e(t_{n-1})$ is defined,
one may discretize \eqref{transport_new} in time using, for example, the implicit Euler method:
\begin{equation}
\frac{u^{n}-u^e(t_{n-1})}{\Delta t} + \bw^{n} \cdot \nabla u^{n} + ({\Div}_\Gamma\bw^{n})u^{n} -{ \eps_d}\Delta_{\Gamma} u^{n}=f^n\quad\text{on}~\Gamma(t_n),
\label{transportFD}
\end{equation}
$\Delta t=t_n-t_{n-1}$.
Now one applies the TraceFEM to discretize \eqref{transportFD} in space: Find $u^{n}_h\in V_h^\Gamma$ satisfying
\begin{equation} \begin{split}  & \int_{\Gamma_h^{n}}\left(\frac{1}{\Delta t}u^{n}_hv_h - (\bw^{n}_h \cdot \nabla v_h) u^{n}_h \right)\,\rd s_h +{ \eps_d}\int_{\Gamma_h^{n}}\nabla u^{n}_h \cdot \nabla v_h\,\rd s_h  \\ & =\int_{\Gamma_h^{n}}\left(\frac{1}{\Delta t}u^{e,n-1}_h+f^n\right)v_h\, \rd s_h
\label{transportFDFE}
\end{split} \end{equation}
for all $v_h\in V_h^\Gamma$. Here $u^{e,n-1}_h$ is a suitable extension of $u^{n-1}_h$ from $\Gamma_h^{n-1}$ to the surface neighborhood, $\mathcal{O}(\Gamma_h^{n-1})$, such that
$\Gamma^{n}_h\subset \mathcal{O}(\Gamma^{n-1}_h).
$
This is not a Courant condition on $\Delta t$, but rather a condition on a width of a strip surrounding
the surface, where the extension of the finite element solution is performed.
A numerical extension procedure, $u^{k}_h\to u^{e,k}_h$,  and the identity \eqref{transportFDFE} define the fully discrete numerical method.
To find a suitable extension, one can consider a numerical solver for hyperbolic systems and apply it to the second equation in \eqref{transport_new}.  Numerical results from \cite{olshanskii2016trace} suggest that the Fast Marching Method \cite{sethian1996fast} is suitable for building suitable extensions in narrow bands of tetrahedra containing $\Gamma_h$, but other (higher order) numerical  methods can be also used.

A potential advantage of the hybrid  TraceFEM is that the TraceFEM for a PDE on a \textit{steady} surface and a hyperbolic solver, e.g., FMM, are used in a modular way. This  makes the implementation straightforward in a standard  finite element software. This variant also  decouples the application of a spatial TraceFEM from the numerical integration in time. The accuracy of the latter can be increased using standard finite differences, while to increase the  accuracy in space one can
consider isoparametric TraceFEM from section~\ref{sectGammaapprox}.
In a series of numerical  experiments using the BDF2 scheme in time and trace $P_1$  finite elements for spatial discretization,  the method demonstrated a second order convergence in space--time and the ability  to handle a surface with topological changes. Stability and convergence analysis of the method is currently an open problem.

\ \\

\noindent
{\bf Acknowledgements.} The authors acknowledge the contributions of A. Cher\-ny\-shen\-ko, A. Demlow, J. Grande, S. Gross, C. Lehrenfeld, and X. Xu to the research topics treated in this article.
\bibliographystyle{abbrv}
\bibliography{literatur}

\end{document}